%% file: Llocalization.tex
\documentclass[10pt, reqno]{amsart}
\usepackage[utf8]{inputenc}  % to allow unicode in source

\usepackage[all]{xy}
\usepackage{amsthm}
\usepackage{amsfonts}
\usepackage{amsmath}
\usepackage{amssymb}

\usepackage{dsfont}
\usepackage{marvosym}
\usepackage{mathpartir}
\usepackage{mathtools} % for \mathclap, used to center extra-wide diagrams.
\usepackage{stmaryrd}

\usepackage{xcolor}
\definecolor{darkgreen}{rgb}{0,0.45,0}
\definecolor{darkred}{rgb}{0.75,0,0}
\definecolor{darkblue}{rgb}{0,0,0.6}
\usepackage[colorlinks,citecolor=darkgreen,linkcolor=darkred,urlcolor=darkblue]{hyperref}
\usepackage[mathscr]{eucal}
\usepackage{enumerate} % for customising enumerated lists
\usepackage{rotating}
\usepackage[capitalize]{cleveref}

\usepackage{mathpartir}

\usepackage{tikz}
\usetikzlibrary{arrows}
\usepackage[all]{xy}
\xyoption{2cell}
\xyoption{curve}
\UseTwocells
\input{diagxy}  % for better inline arrows

\input{commands1}
\begin{document}
\title{$L'$-localization in an $\infty$-topos}

\author{Marco Vergura}
\email{mvergura@uwo.ca}
\date{8 July 2019}

\begin{abstract}
We prove that, given any reflective subfibration $L_\bullet$ on an $\infty$-topos $\E$, there exists a reflective subfibration $L'_\bullet$ on $\E$ whose local maps are the $L$-separated maps, that is, the maps whose diagonals are $L$-local. This is the companion paper to \cite{locinaninftytopos}.
\end{abstract}
\maketitle
\tableofcontents
\section{Introduction}
This paper complements the work of \cite{locinaninftytopos} by proving the following theorem, which is one of our main results in the theory of reflective subfibrations on an $\infty$-topos $\E$.
\begin{theoremstar}[\cref{thm:existoflprimeloc} \& \cref{cor:l'reflsubf}]
Let $L_\bullet$ be a reflective subfibration on an $\infty$-topos $\E$. Then, there exists a reflective subfibration $L'_\bullet$ on $\E$ for which the $L'$-local maps are exactly the $L$-separated maps.
\end{theoremstar}
\par

In \cite{locinaninftytopos}, we took from \cite{modinhott} the notion of \emph{reflective subfibration} on an $\infty$-topos $\E$, and developed the study of its properties. A reflective subfibration $L_\bullet$ on $\E$ is a pullback-compatible system of reflective subcategories $\D_{X}$ of $\E_{/X}$, for every $X\in\E$, with associated localization functor denoted by $L_X$. The collection of all objects in $\D_X$, as $X$ varies in $\E$, forms the class of $L$\emph{-local maps}. Reflective subfibrations provide a suitable framework for the study of localizations in an $\infty$-topos. Indeed, all the most common examples of localizations from classical homotopy theory can be recovered in this setting: stable factorization systems (\cite[Thm.~4.8]{locinaninftytopos}), left exact reflective subcategories of an $\infty$-topos (\cite[Prop.~4.11]{locinaninftytopos}), and localizations at sets of maps (\cite[Prop.~5.11]{locinaninftytopos}). For the reader's convenience, in \cref{sec:reflsub}, we briefly gather from \cite{locinaninftytopos} the main definitions and results about reflective subfibrations that we need here.\par

For a reflective subfibration $L_\bullet$ on $\E$, one can consider the $L$\emph{-separated maps}, that is, those maps in $\E$ whose diagonal is $L$-local. For example, for the reflective subfibration $L^n_\bullet$ having the $n$-truncated maps as local maps, the $L^n_\bullet$-separated maps are the $(n+1)$-truncated maps, and they are themselves the local maps for a reflective subfibration, $L^{n+1}_\bullet$. It turns out that this behavior is completely general, as shown by \cref{thm:existoflprimeloc} and \cref{cor:l'reflsubf}: for \emph{any} $L_\bullet$, there exists an $L'_\bullet$ such that the $L'$-local maps are the $L$-separated maps.\par

In this paper, we focus on the proof of this existence result, leaving the study of its consequences to \cite[\S 7]{locinaninftytopos}. To this end, one needs to carefully examine some connections between $L$-local and $L$-separated maps. We develop the study of these relationships in \cref{sec:relbtwnllocandlsep}. Our main result there is the following characterization of $L'$\emph{-localization maps}, that is, those maps out of a fixed object $X$ (or, more generally, out of a map $p$) and into an $L$-separated object, which are universal among maps with this property.
\newtheorem*{thm:charoflprimeloc}{\cref{thm:charoflprimeloc}}
\begin{thm:charoflprimeloc}
The following are equivalent, for a map $\eta'\colon X\rightarrow X'$ in $\E$:
\begin{enumerate}
\item $\eta'$ is an $L'$-localization of $X$;
\item $\eta'$ is an effective epimorphism and
$$
\bfig
\node x(0,0)[X]
\node xx(300,-300)[X\times X]
\node xx'x(600,0)[X\times_{X'}X]

\arrow|l|[x`xx;\Delta X]
\arrow|a|[x`xx'x;\Delta\eta']
\arrow||[xx'x`xx;]
\efig
$$
is an $L$-localization of $\Delta X$.
\end{enumerate}
\end{thm:charoflprimeloc}
\par

The existence result for $L'_\bullet$, together with a few auxiliary lemmas needed in its proof, is the content of \cref{sec:existofl'loc}. The results in both \cref{sec:relbtwnllocandlsep} and \cref{sec:existofl'loc} require some facts about locally cartesian closed $\infty$-categories that we collect in the Appendix (\cref{AppA}). Some of the results there are well known, but for others we could not find any reference in the literature. Examples of the results in the latter group are \cref{prop:funext}, where we prove the topos-theoretic version of the function extensionality axiom from HoTT, and \cref{prop:uniqueextalongfibers}, which provides a criterion for unique extensions of maps that is crucial for the proof of \cref{thm:charoflprimeloc}.\par

Our approach to localization is inspired by the work in homotopy type theory (HoTT) developed in \cite{locinhott}. The notion of $L$-separated map, as well as \cref{prop:extoflocalalongL'loc} and \cref{thm:existoflprimeloc}, are expressed in HoTT in \cite[\S 2.2-2.3]{locinhott}. We take from there the main ideas for the proofs of \cref{thm:charoflprimeloc} and \cref{thm:existoflprimeloc}. However, proof details and techniques have been modified, sometimes significatively, to apply to the ``term-free" exposition we work with. This is particularly evident in the proof of \cref{thm:charoflprimeloc}, and in the results of \cref{sec:existofl'loc}. All the proofs of the results in the Appendix are also specific to the higher-topos theoretic setting we work with. For a more detailed description of how our work relates to the study of localization in HoTT, we refer the reader to the Introduction of \cite{locinaninftytopos}.

\subsection*{Acknowledgements} We would like to thank Dan Christensen, for his support and guidance, and Mike Shulman, for the careful reading of the material present here, and for many helpful suggestions.

\subsection*{Notation and Conventions.} Notation and conventions from \cite{locinaninftytopos} carry over here as well. Furthermore, given an $\infty$-category $\C$, we often depict a map $m\colon p\rightarrow q$ in a slice category $\C_{/Z}$ as a commuting triangle in $\C$ of the form
$$
\bfig
\node e(0,0)[E]
\node m(500,0)[M]
\node z(250,-250)[Z]

\arrow|a|[e`m;m]
\arrow|l|[e`z;p]
\arrow|r|[m`z;q]
\efig
$$
leaving the interior $2$-simplex implicit. We will often carry over this implicitness to other maps in slice categories that are constructed from $m$, at least as long as the context is enough to disambiguate. For example, if the implicit $2$-simplex of $m$ above is $\sigma$, then $(\sigma,\sigma)$ is the implicit $2$-simplex of the map in $\C_{/Z^2}$ given by
$$
\bfig
\node e(0,0)[E]
\node m(500,0)[M]
\node z(250,-250)[Z^2]

\arrow|a|[e`m;m]
\arrow|l|[e`z;(p,p)]
\arrow|r|[m`z;(q,q)]
\efig
$$
\par
If $p$ and $q$ are objects in a slice category $\C_{/Z}$, we write $p\times^{Z}q$ to mean the product object of $p$ and $q$ in $\C_{/Z}$.
\section{Reflective Subfibrations}
\label{sec:reflsub}
We gather here some background material on reflective subfibrations in an $\infty$-topos $\E$ from the companion paper \cite{locinaninftytopos}.
\begin{definition}[{\cite[\S A.2]{modinhott}}]
\label{def:refsub}
Let $\E$ be an $\infty$-topos.
\begin{enumerate}
\item A \emph{reflective subfibration} $L_\bullet$ on $\E$ is the assignment, for each $X\in\E$, of an $\infty$-category $\D_X$ such that:
\begin{itemize}
\item[(a)] Each $\D_X$ is a reflective $\infty$-subcategory of $\E_{/X}$, with associated localization functor $L_X=\colon \E_{/X}\rightarrow\E_{/X}$. This is the composite of the reflector of $\E_{/X}$ into $\D_X$ and the inclusion of $\D_X$ into $\E_{/X}$. When $X=1$, we write $\D$ for $\D_1$ and $L$ for $L_1$. 
\item[(b)] For every map $f\colon X\rightarrow Y$ in $\E$, and any $p\in \E_{/Y}$, the induced map $L_X(f^\ast p)\rightarrow f^\ast(L_Y p)$ is an equivalence. In particular, the pullback functor $f^\ast\colon\E_{/Y}\rightarrow\E_{/X}$ restricts to a functor $\D_Y\rightarrow\D_X$ which we still denote by $f^\ast$.
\end{itemize}
\item A \emph{modality} on $\E$ is a reflective subfibration $L_{\bullet}$ on $\E$ which is composing, in that, whenever $p\colon X\rightarrow Y$ is in $\D_Y$ and $q\colon Y\rightarrow Z$ is in $\D_Z$, the composite $qp$ is in $\D_Z$.
\end{enumerate}  
\end{definition}
\begin{remark}
\label{rmk:localcharacterofreflsubf}
For every object $X\in\E$ and every map $f\colon Y\rightarrow X$, we have that $(\E_{/X})_{/f}\simeq \E_{/Y}$ (see \cite[Lemma~4.18]{joyalconjhott}). Therefore, for each $X\in\E$, a reflective subfibration $L_{\bullet}$ induces a reflective subfibration $L^{/X}_{\bullet}$ of $\E_{/X}$ by taking $\D^{/X}_{f}$ to be $\D_Y$. It follows that all the results we give below about reflective subfibrations on an $\infty$-topos also hold ``locally" in the $\infty$-topos $\E_{/X}$, for $X\in\E$.
\end{remark}
From now on, we fix a reflective subfibration $L_{\bullet}$ on our favorite $\infty$-topos $\E$. 

\begin{notation}\label{not:stuffaboutreflsub} We adopt the following notation for the rest of this work.

\begin{itemize}
\item A morphism $p\colon E\rightarrow X$ is called $L$-\emph{local} if, seen as an object of $\E_{/X}$, it is in $\D_X$. We call $E\in\E$ an $L$-\emph{local object} if $E\rightarrow 1$ is an $L$-local map. 
\item For $X\in\E$, $S_X$ denotes the class of all $L_X$-\emph{equivalences}, i.e., maps $\alpha$ in $\E_{/X}$ such that $L_X(\alpha)$ is an equivalence. Equivalently, $S_X=\prescript{\perp}{}{\D_X}$, where $\prescript{\perp}{}{\D_X}$ denotes the class of maps in $\E_{/X}$ which are left orthogonal to maps in $\D_X$. When it is clear that $\alpha$ is a map in $\E_{/X}$, we often drop the explicit reference to the object $X$, and just talk about $L$-\emph{equivalences}. 
\item Given $p\in\E_{/X}$, we write $\eta_X(p)\colon p\rightarrow L_X(p)$ for the reflection (or localization) map of $p$ into $\D_X$. Note that $\eta_X(p)\in S_X$. For $X\in\E$, we set $\eta(X):=\eta_1(X)$. 
\end{itemize}
\end{notation}

Given a map $f$ in $\E$, we denote by $\Sigma_f$ and by $\Pi_f$ the left and right adjoint to the pullback functor $f^\ast$, respectively.
\begin{lemma}
\label{lm:lequivandpullbdepsum}
Given $f\colon X\rightarrow Y$, we have:
\begin{itemize}
\item[(i)]$f^\ast(S_Y)\subseteq S_X$, that is, if $\alpha\colon p\rightarrow q$ is an $L_Y$-equivalence, then the induced map $f^\ast(p)\rightarrow f^\ast (q)$ on pullbacks is an $L_X$-equivalence;
\item[(ii)]$\Sigma_f(S_X)\subseteq S_Y$. 
\end{itemize}
\end{lemma}

In \cite[\S 2]{locinaninftytopos}, we introduced the notion of local class of maps and of univalent classifying maps in $\E$, basing on \cite{htt} and \cite{univinloccarclosed}. Recall, in particular, that there are arbitrarily large regular cardinals such that there is a univalent map $u_\kappa\colon\ti{\U_\kappa}\rightarrow \U_\kappa$ classifying $\kappa$-compact maps in $\E$.
\begin{proposition}\cite[Prop.~3.12 \& Thm.~3.15]{locinaninftytopos}
\label{prop:llocalmapsarelocal}
The class $\M^L$ of all $L$-local maps is a  local class of maps of $\E$. In particular, there are abritrarily large regular cardinals $\kappa$ such the class of relatively $\kappa$-compact $L$-local maps admits  a classifying map
$u_\kappa^L\colon\ti{\U_\kappa^L}\rightarrow \U_\kappa^L$ which is univalent.
\end{proposition}

\begin{definition}
\label{def:lconnmaps}
$f\in\E_{/X}$ is said to be an $L$-\emph{connected map (in} $\E$\emph{)} if $L_X(f)\simeq \id_X$. Equivalently, $f$ is $L$-connected if 
$$(f\stackrel{\eta_X(f)}{\longrightarrow}L_X(f))\simeq (f\stackrel{f}{\rightarrow}\id_X)$$
in the arrow category of $\E_{/X}$, where the equivalence is given by $\id_f$ and $L_X(f)\rightarrow~\id_X$. We sometimes refer to this fact by saying that an $L$-connected map $f$ is \emph{its own reflection map}. 
\end{definition}

In particular, an $L$-connected map $f\colon E\rightarrow X$ is an $L_X$-equivalence when seen as a map $f\colon f\rightarrow \id_X$ in $\E_{/X}$.

\begin{remark}
\label{rmk:lconnmapsclosedunderpullbacks}
By taking the reflection of $f\in\E_{/X}$ into $\D_X$ and using stability under pullbacks of reflection maps (see \cref{def:refsub} (1a)), it follows that $L$-connected maps are stable under pullbacks along arbitrary maps.
\end{remark}
We now recall the definition of an $L$-separated map, which is the core notion in this paper.
\begin{definition}
\label{def:lseparatedmaps}
A map $p\colon E\rightarrow X$ in $\E$ is called $L$-\emph{separated} or $L'$-\emph{local} if the object $\Delta p\in\E_{/E\times_{X}E}$ is in $\D_{E\times_{X}E}$, i.e., if $\Delta p$ is an $L$-local map.
\end{definition}
\begin{proposition}[{\cite[Prop.~6.5 \& Prop.~6.7]{locinaninftytopos}}]
\label{prop:closureoflsepmapsunderpllbckdepprod}
Let $L_\bullet$ be a reflective subfibration on an $\infty$-topos $\E$. Then the following hold.
\begin{enumerate}
\item Let $f\colon Y\rightarrow X$ be a map in $\E$, and let $p\colon E\rightarrow X$ and $q\colon M\rightarrow~Y$ be $L$-separated maps. Then $f^{\ast}(p)\in\E_{/Y}$ and $\prod_f q\in\E_{/X}$ are $L$-separated. Furthermore, the internal hom $p^f$ is $L$-separated.
\vspace*{0.1cm}
\item The class $\M'$ of all $L$-separated maps is a local class of maps.
\end{enumerate}

\end{proposition}
\section{Interactions between $L$-local and $L$-separated maps}\label{sec:relbtwnllocandlsep}
We study here some relationships between $L$-local and $L$-separated maps and prove a characterization result for $L'$-localization maps which will be used in the next section as a fundamental step for the proof of \cref{thm:existoflprimeloc}.
\begin{lemma}[{\cite[Lemma 2.21]{locinhott}}]
\label{lm:totalspaceofllocoverlsepislloc}
Suppose given a commutative triangle
$$
\bfig
\Vtriangle|alr|<300,250>[E`M`X;\alpha`p`q]
\efig
$$
in which $\Delta q\in\D_{M\times_X M}$ and $\alpha\in\D_M$, that is, $q$ is $L$-separated and $\alpha$ is $L$-local. Then $\Delta p$ is in $\D_{E\times_X E}$, i.e., $p$ is $L$-separated.
\end{lemma} 
\begin{proof}
The map
$
(\id_E\times_X\alpha\colon E\times_X E\rightarrow E\times _X M)=(E\times_X M\rightarrow M)^{\ast}(\alpha)
$ 
is in $\D_{E\times_X M}$, since $\alpha\colon E\rightarrow M$ is in $\D_M$. Similarly, the map
$
((\id_E,\alpha)\colon E\rightarrow E\times_X M)=(\alpha\times _X \id_M)^\ast (\Delta q)
$
is in $\D_{E\times_X M}$. But $(\id_E\times_X\alpha)\circ \Delta p = (\id_E,\alpha),$ so $\Delta p$ is $L$-local, by \cite[Prop.~3.7]{locinaninftytopos}: if both $f$ and $f\circ g$ are $L-$local maps, then so is $g$.
\end{proof}

\begin{definition}
\label{def:l'locmap}
A map $\alpha\colon p\rightarrow p'$ in $\E_{/X}$ is called an $L'$-\emph{localization map} of $p$ if $p'$ is $L$-separated and 
$
\E_{/X}(\alpha,q)\colon\E_{/X}(p',q)\rightarrow \E_{/X}(p,q)
$
is an equivalence of $\infty$-groupoids for every $L$-separated $q\in\E_{/X}$. In other words, for every map $\beta\colon p\rightarrow q$, there is a unique $\psi\colon p'\rightarrow q$ with $\psi\circ\alpha =\beta$.
\end{definition}
\begin{remark}
Given an $L-$separated $r\in\E_{/X}$ and any $t\in\E_{/X}$, $r^t\in\E_{/X}$ is again $L$-separated. It follows that, for a map $\alpha\colon p\rightarrow p'$ in $\E_{/X}$ with $p'$ $L$-separated, the above definition can be rephrased internally, by asking that $q^\alpha$ is an equivalence in $\E_{/X}$ for every $L$-separated map $q\colon Y\rightarrow X$.
\end{remark}

\begin{lemma}[{\cite[Prop.~ 2.30]{locinhott}}]
\label{lm:etaprimelconn}
Let $\eta '\colon p\rightarrow p'$ in $\E_{/Y}$ be an $L'$-localization of $p\in\E_{/Y}$, with $\eta '\colon X\rightarrow~X'$ as a map in $\E$. Then $\eta'$ is an $L$-connected map (\cref{def:lconnmaps}).
\end{lemma}
\begin{proof}
Let $\eta_{X'}(\eta ')\colon \eta'\rightarrow L_{X'}(\eta')$ be the reflection map of $\eta'\in\E_{/X'}$ into $\D_{X'}$. Set $r:= p'\circ L_{X'}(\eta')$, and consider $\eta_{X'}(\eta ')\colon p\rightarrow r$ and $L_{X'}(\eta')\colon r\rightarrow p'$ as maps in $\E_{/Y}$. By \cref{lm:totalspaceofllocoverlsepislloc} applied to $L_{X'}(\eta')$, $r$ is $L$-separated. Hence, there is a unique $q\colon p'\rightarrow~r$ with $q\eta' =\eta_{X'}(\eta ')$ as maps $p\rightarrow r$ in $\E_{/Y}$. Since $L_{X'}(\eta')q\eta'=L_{X'}(\eta')\eta_{X'}(\eta ')=\eta',$ 
the universal property of $\eta'$ gives $L_{X'}(\eta')q=\id_{p'}$. Thus, we can consider $qL_{X'}(\eta')$ as a map $L_{X'}(\eta')\rightarrow L_{X'}(\eta')$ in $\E_{/X'}$ and
$qL_{X'}(\eta')\eta_{X'}(\eta ')=q\eta'=\eta_{X'}(\eta '),$
so that $qL_{X'}(\eta')=\id_r$. Hence, $\eta'$ is $L$-connected.
\end{proof}
\begin{lemma}
\label{lm:ULisLsep}
Let $\kappa$ be a regular cardinal such that the class of relatively $\kappa$-compact $L$-local maps has a classifying map $u^L_\kappa\colon \ti{\U^L_\kappa}\rightarrow \U^L_\kappa$. Then $\U^{L}_\kappa$ is $L$-separated.
\end{lemma}
\begin{proof}
We drop $\kappa$ from our notation. Since $u^L$ is univalent (\cref{prop:llocalmapsarelocal}), we have an equivalence
$
\Delta (\U^L)\simeq \Eq_{/ \U^L}(u^L)
$
over $\U^L\times\U^L$. By definition, $\Eq_{/ \U^L}(u^L)$ is the object of equivalences in $\E_{/\U^L\times \U^L}$ between $\id_{\U^L}\times u^L$ and $u^L\times \id_{\U^L}$, both of which are $L$-local since $u^L$ is. By \cite[Lemma 2.8]{locinaninftytopos}, such an object of equivalences is then the pullback of a cospan of objects in $\D_{\U^L\times\U^L}$ and it is therefore in $\D_{\U^L\times\U^L}$.   
\end{proof}
\begin{proposition}
\label{prop:extoflocalalongL'loc}
Let $X\in\E$ and let $\eta'\colon X\rightarrow X'$ be an $L'$-localization of $X$. Then a map $p\colon E\rightarrow X$ is $L$-local if and only if there is a pullback square in $\E$
$$
\bfig
\square<600,250>[E`L_{X'}E`X`X';\eta_{X'}(\eta'p)`p`L_{X'}(\eta'p)`\eta']
\efig
$$
\end{proposition}
\begin{proof}
For the non-trivial implication, assume $p$ is $L$-local. Let $\kappa$ be a regular cardinal such that $p$ is relatively $\kappa$-compact and the class of relatively $\kappa$-compact $L$-local maps has a classifying map $u^L\colon\ti{\U^L_\kappa}\rightarrow \U^L_\kappa$. Let $P\colon X\rightarrow \U^L_\kappa$ be such that we have a pullback square
\begin{equation}
\tag{$\dagger$}
\bfig
\square<600,400>[E`\ti{\U^L_\kappa}`X`\U^L_\kappa;`p`u^L`P]
\place(70,330)[\angle]
\efig
\end{equation}
Since $\U^L_\kappa$ is $L$-separated, there is a unique map $P'\colon X'\rightarrow \U^L_\kappa$ with $P = P'\eta'$. Let $p'\colon E'\rightarrow X'$ be the pullback map in  
$$
\bfig
\square<600,250>[E'`\ti{\U^L_\kappa}`X'`\U^L_\kappa;`p'`u^L`P']
\place(70,180)[\angle]
\efig
$$
By definition of $P'$, $\eta'\colon X\rightarrow X'$ induces a map $n\colon E\rightarrow E'$ such that the composite square in
\begin{equation}
\tag{$\ddagger$}
\bfig
\hSquares(0,0)|aallrbb|%
<250>[E`E'`\ti{\U^L_\kappa}`X`X'`\U^L_\kappa;n``p`p'`u^L`\eta'`P']
\place(600,180)[\angle]
\efig
\end{equation}
is the square $(\dagger)$. It follows that the left square in $(\ddagger)$ is also a pullback. Thanks to \cref{lm:etaprimelconn}, $\eta'$ is $L$-connected. Thus, so is $n$, by \cref{rmk:lconnmapsclosedunderpullbacks}. In particular, $n$ is an $L$-equivalence (i.e., $n\colon n\rightarrow \id_{E'}$ is in $S_{E'}$). By composing domain and codomain of $n\colon n\rightarrow \id_E$ with $p'$, \cref{lm:lequivandpullbdepsum} (ii) gives that $n\colon \eta'p\rightarrow p'$ is an $L$-equivalence. Since $p'$ is $L$-local, it follows that $n$ is the $L$-localization map of $\eta'p$, as required.
\end{proof}
\begin{remark}
As explained in \cref{rmk:localcharacterofreflsubf}, \cref{prop:extoflocalalongL'loc} is also true ``locally'', i.e., when we take our ground $\infty$-topos to be $\E_{/X}$ instead of $\E$. For the result above, this means specifically that, if  
$$
\bfig
\Vtriangle|alr|<400,250>[E`E'`X;\eta'_{X}(p)`p`p']
\efig
$$
is an $L'$-localization of $p$ in $\E_{/X}$, a map 
$$
\bfig
\Vtriangle|alr|<400,250>[Y`E`X;m`q`p]
\efig
$$
is $L^{/X}$-local (as an object in $(\E_{/X})_{/p}$, so $m$ is in $\D_E$) if and only if
$$
\bfig
\square<700,250>[Y`L_{E'}Y`E`E';\eta_{E'}(\eta'_{X}(p)\circ m)`m`L_{E'}(\eta'_{X}(p)m)`\eta'_{X}(p)]
\efig
$$
is a pullback square in $\E_{/X}$. (Note that, in the above, $L_{E'}$ should be $L^{/X}_{p'}$, where $L^{/X}_{p'}$ is the reflector of $(\E_{/X})_{/p'}$ onto $\D^{/X}_{p'}$ and $L^{/X}_\bullet$ is the reflective subfibration on $\E_{/X}$ induced by $L_\bullet$, as in \cref{rmk:localcharacterofreflsubf}. But, by its own definition, $L^{/X}_{p'}=L_{E'}$.)
\end{remark}
The following corollary is probably well-known, though the only explicit reference we could find in the literature is \cite[Lemma 8.6]{topandhomtop}, where the statement is proved in the context of model topoi. Note that our proof is completely internal and does not use the description of $\infty$-topoi as left exact localizations of presheaf categories.

\begin{corollary}
\label{cor:charofntruncmaps}
For $n\geq -2$, a map $p\colon E\rightarrow X$ is $n$-truncated if and only if $\norm{p}_{n+1}$ is $n$-truncated and there is a pullback square
$$
\bfig
\square<500,300>[E`\norm{E}_{n+1}`X`\norm{X}_{n+1};\abs{\cdot}_{n+1}`p`\norm{p}_{n+1}`\abs{\cdot}_{n+1}]
\place(70,220)[\angle]
\efig
$$
\end{corollary}
\begin{proof}
By \cite[Ex.~4.6 \& Ex.~6.4]{locinaninftytopos}, we can apply \cref{prop:extoflocalalongL'loc} where $L_\bullet$ is the $n$-truncation modality and get a pullback square
$$
\bfig
\square<600,300>[E`L_{\norm{X}_{n+1}}(E)`X`\norm{X}_{n+1};n`p`L_{\norm{X}_{n+1}}(\abs{\cdot}_{n+1}p)`\abs{\cdot}_{n+1}]
\place(70,220)[\angle]
\efig
$$
Since $\norm{X}_{n+1}$ is $(n+1)$-truncated and $L_{\norm{X}_{n+1}}(\abs{\cdot}_{n+1}p)$ is $n$-truncated, $L_{\norm{X}_{n+1}}(E)$ is $(n+1)$-truncated. (This is an instance of \cref{lm:totalspaceofllocoverlsepislloc}.) But $n$ is a pullback of the $(n+1)$-connect\-ed map $\abs{\cdot}_{n+1}\colon X\rightarrow\norm{X}_{n+1}$, so it is $(n+1)$-connected. Finally, any $(n+1)$-connected map $m\colon A\rightarrow B$ where $B$ is $(n+1)$-truncated is an $(n+1)$-truncation map of $A$.
\end{proof}
\begin{proposition}[{\cite[Prop.~ 2.26]{locinhott}}]
\label{prop:locofdiagonalandkernelpairofetaprime}
Let
$$
\bfig
\Vtriangle|alr|<350,250>[E`E'`X;\eta'_X(p)`p`p']
\efig
$$
be an $L'$-localization of $p\in\E_{/X}$. Let
$$
\bfig
\Vtriangle|alr|<350,300>[E`R`E\times_{X} E;\eta_{E\times_X E}(\Delta p)`\Delta p`r]
\efig
$$ 
be the $L$-localization of $\Delta p\in \E_{/E\times_{X} E}$ and consider $r'$ defined by the pullback square
\begin{equation}
\tag{$\dagger$}
\bfig
\square<1000,250>[E\times_{E'}E`E'`E\times_{X} E`E'\times_{X} E';`r'`\Delta p'`\eta'_X(p)\times_{X}\eta'_X(p)]
\place(100,150)[\angle]
\efig
\end{equation}
Then there is a natural equivalence $\phi\colon R\stackrel{\simeq}{\rightarrow} E\times_{E'}E$ over $E\times_{X}E$ as in
$$
\bfig
\morphism(0,0)<-350,-300>[E`R;\eta_{E\times_X E}(\Delta p)]
\morphism(0,0)<350,-300>[E`E\times_{E'} E;\Delta(\eta'_X(p))]
\morphism(-350,-300)|l|<350,-350>[R`E\times_X E;r]
\morphism(350,-300)|r|<-350,-350>[E\times_{E'} E`E\times_X E;r']
\morphism(-350,-300)|a|<700,0>[R`E\times_{E'} E;\phi]
\efig
$$
\end{proposition}
\begin{proof}
For sake of readability, we write $\eta'$ and $\eta$ for $\eta'_X(p)$ and $\eta_{E\times_X E}(\Delta p)$, respectively. The natural map $\phi$ is given by the universal property of $\eta$, since $r'$ is $L$-local. (By definition, $r'$ is the pullback of the $L$-local map $\Delta p'$.) Now, since $\eta'\times_X \eta'$ is the $L'$-localization map of the product object $p\times^X p$ of $\E_{/X}$, \cref{prop:extoflocalalongL'loc} applied in $\E_{/X}$ gives that there is a pullback square
$$
\bfig
\node r(0,0)[R]
\node t(800,0)[T]
\node exe(0,-250)[E\times_X E]
\node e'xe'(800,-250)[E'\times_X E']

\arrow|a|[r`t;n]
\arrow|l|[r`exe;r]
\arrow|r|[t`e'xe';q]
\arrow|a|[exe`e'xe';\eta'\times_{X}\eta']

\place(70,-70)[\angle]

\efig
$$
where $n\colon (\eta'\times_{X}\eta')r\rightarrow q$ is the $L$-localization map of $(\eta'\times_{X}\eta')r$. Set $m:= n\eta\colon E\rightarrow~T$ and $l:=\pi q$, where $\pi\colon E'\times_{X} E'\rightarrow X$ is given by the composite map $E'\times_{X} E'\rightarrow E'\stackrel{p'}{\rightarrow}~X$. Note that $\pi$ is $L$-separated, because it is the product in $\E_{/X}$ of the $L$-separated map $p'$ with itself. Hence, since $q$ is $L$-local, $l$ is $L$-separated by \cref{lm:totalspaceofllocoverlsepislloc}. Since $m=n\eta$ is naturally a  map $m\colon p\rightarrow l$ in $\E_{/X}$, there is a unique $s\colon E'\rightarrow T$ over $X$ with commuting triangles
$$
\bfig
\morphism(0,0)<-400,-300>[E`E';\eta']
\morphism(0,0)<400,-300>[E`T;m]
\morphism(-400,-300)|l|<400,-300>[E'`X;p']
\morphism(400,-300)|r|<-400,-300>[T`X;l]
\morphism(-400,-300)|a|<800,0>[E'`T;s]
\efig
$$
Now,
$
qs\eta'=qm=qn\eta=(\eta'\times_X \eta')\Delta p=\Delta p'\eta'
$
so that $qs=\Delta p'$ and we can write $s\colon \Delta p'\rightarrow q$ as a map over $E'\times_{X}E'$. Hence, $s$ induces the comparison map $\psi$ of pullback squares in 
$$
\bfig
\cube|alma|<1000,650>[E\times_{E'}E`E'`\ E\times_X E`E'\times_{X}E';`r'`\Delta p'`]%
(350,-270)|alrb|<1000,650>[R`T`E\times_X E`E'\times_X E';n`r`q`\eta'\times_X \eta']%
/{-->}`{>}`{>}`{>}/[\psi`s`\id`\id]
\efig
$$
Since the front face is a pullback, it follows that $\psi\circ\Delta \eta'=\eta$, from which we get $\psi\phi\eta=\eta$, so that $\psi\circ\phi=\id$. We now claim that $s$ is an equivalence. This would imply that $\psi$ (and therefore also $\phi$) is an equivalence. Since $s\colon\Delta p'\rightarrow q$ is a map between $L$-local maps over $E'\times_{X}E'$, it is enough to show that $s\in S_{E'\times_{X}E'}$. Now, $\eta'\colon p\rightarrow p'$ is $L$-connected so it is an $L_{E'}$-equivalence  (more precisely, $\eta'\colon\eta'\rightarrow~\id_{E'}$ is in $S_{E'}$). By \cref{lm:lequivandpullbdepsum}~(ii), composing $\eta'\colon\eta'\rightarrow\id_{E'}$ with $\Delta p'$ gives that $\eta'\colon (\Delta p')\eta'\rightarrow\Delta p'$ is in $S_{E'\times_{X}E'}$. Similarly, composing domain and codomain of $\eta$ with $\eta'\times_X \eta'$ turns $\eta$ into a map in $S_{E'\times_{X}E'}$ and then $m=n\eta$ is in $S_{E'\times_{X}E'}$, since $n$ is an $L$-equivalence. Since $s\eta'=m$, $s\in S_{E'\times_{X}E'}$, as needed.
\end{proof}

Our next result characterizes $L'$-localization maps in terms of their diagonal maps. We will use here some results from the Appendix (\cref{subsec:contractibility}).

\begin{theorem}[{\cite[Thm.~2.34]{locinhott}}]
\label{thm:charoflprimeloc}
The following are equivalent for a map in $\E_{/Z}$
$$
\bfig
\Vtriangle|alr|<400,300>[X`X'`Z;\eta'`p`p']
\efig
$$
\begin{enumerate}
\item $\eta'$ is an $L'$-localization map of $p$.
\item $\eta'$ is an effective epimorphism and
$$
\bfig
\Vtriangle|alr|<400,300>[X`X\times_{X'}X`X\times_Z X;\Delta\eta'`\Delta p`]
\efig
$$
is an $L$-localization map of $\Delta p$.
\end{enumerate}
\end{theorem}
\begin{proof}
We prove the theorem when $Z=1$; the general statement follows from this one by \cref{rmk:localcharacterofreflsubf}. We show first that (1) implies (2). If $\eta'\colon X\rightarrow X'$ is an $L'$-localization of $X$, then, by \cref{prop:locofdiagonalandkernelpairofetaprime},
we only need to show that $\eta'$ is an effective epimorphism. Let $(\pi,i)$ be the (effective epi,mono)-factorization of $\eta'$, with $i\colon W\rightarrow X'$. Since $i$ is a mono, $\Delta W= (i\times i)^\ast (\Delta X')$. Hence, since $X'$ is $L$-separated, so is $W$. Therefore there is a unique $s\colon X'\rightarrow W$ with $s\eta'=\pi$. From $is\eta'=i\pi=\eta'$, we get that $is=\id_{X'}$. Thus, $i$ is both a mono and an effective epi, so it is an equivalence.\par
Conversely, assume $\eta'$ is an effective epimorphism and $\Delta \eta'$ is the $L$-localization of $\Delta X$. In the pullback square
\begin{equation}
\label{eq:tcharofl'}
\tag{$\ast$}
\bfig
\square<600,250>[X\times_{X'}X`X'`X\times X`X'\times X';`t`\Delta X'`\eta'\times \eta']
\place(70,150)[\angle]
\efig
\end{equation}
$\eta'\times\eta'$ is also an effective epi and $t$ is $L$-local by hypothesis. Thus, $\Delta X'$ is $L$-local since $L$-local maps are a local class of maps in $\E$. This shows that $X'$ is $L$-separated. We now verify that $\eta'$ has the universal property of an $L'$-localization map. Let $f\colon X\rightarrow Y$ be a map into an $L$-separated object $Y$. We show that $f$ extends uniquely along $\eta'$, by applying \cref{prop:uniqueextalongfibers} to $f$ and $\eta'$. We want to show that
$$
E:=\sum_{X'\times Y\rightarrow X'}\left(\prod_{X\times X'\times Y\rightarrow X'\times Y}(\pr_X,X'\times f)^{(\pr_X,\eta'\times Y)}\right)
$$
is contractible in $\E_{/X'}$. Applying \cref{lm:iscontrandpllbck} and the Beck-Chevalley condition (\cref{lm:BCcond}) to the pullback squares
$$
\bfig
\node xxy(0,0)[X\times X\times Y]
\node xy(700,0)[X\times Y]
\node x(1400,0)[X]
\node xx'y(0,-300)[X\times X'\times Y]
\node x'y(700,-300)[X'\times Y]
\node x'(1400,-300)[X']

\arrow[xxy`xy;\pr_{X\times Y}]
\arrow[xy`x;\pr_{X}]
\arrow[xxy`xx'y;X\times\eta'\times Y]
\arrow[xx'y`x'y;]
\arrow|b|[xy`x'y;\eta'\times Y]
\arrow|m|[x'y`x';]
\arrow|b|[x`x';\eta']

\place(100,-100)[\angle]
\place(800,-100)[\angle]

\efig
$$
we can instead show that
$$
E':=\sum_{X\times Y\rightarrow X}\left(\prod_{X\times X\times Y\rightarrow X\times Y}(X\times\eta'\times Y)^\ast\left((\pr_X,X'\times f)^{(\pr_X,\eta'\times Y)}\right)\right)
$$
is contractible in $\E_{/X}$. We will show that this object of $\E_{/X}$ is equivalent to the object $\id_X$, which is contractible in $\E_{/X}$. \cref{lm:pllbckcartclosedfunct} gives that
$$
(X\times\eta'\times Y)^\ast\left((\pr_X,X'\times f)^{(\pr_X,\eta'\times Y)}\right)\simeq$$$$\simeq\left((X\times\eta'\times Y)^\ast(\pr_X,X'\times f)\right)^{(X\times\eta'\times Y)^\ast((\pr_X,\eta'\times Y))}
$$
Notice that $$(\pr_X,X'\times f)=(f\times \pr_Y)^\ast(\Delta Y),\quad (\pr_X,\eta'\times Y)=(\eta'\times\pr_{X'})^\ast(\Delta X')$$ and $(f\times \pr_Y)(X\times\eta'\times Y)=(f\times Y)(\pr_1,\pr_3),$ where $\pr_1\colon X\times X\times Y\rightarrow X$ and $\pr_3\colon X\times X\times Y\rightarrow Y$ are appropriate projections. One can then see that
$$
(X\times\eta'\times Y)^\ast\left((\pr_X,X'\times f)\right)=(\id_{X\times X}, f\pr_1)\colon X\times X\rightarrow X\times X\times Y,
$$
$$
(X\times\eta'\times Y)^\ast\left((\pr_X,\eta'\times Y)\right)=t\times Y\colon (X\times_{X'}X)\times Y\rightarrow X\times X\times Y
$$
where $t$ is defined in the pullback square \eqref{eq:tcharofl'} above. Therefore, $$
(X\times\eta'\times Y)^\ast\left((\pr_X,X'\times f)^{(\pr_X,\eta'\times Y)}\right)\simeq (\id_{X\times X}, f\pr_1)^{t\times Y}.$$
Now, since $t$ is the localization of $\Delta X$ in $\E_{/X\times X}$, taking pullbacks along the projection $X\times X\times Y\rightarrow X\times X$ gives that $t\times Y$ is the localization of $\Delta X\times Y$ in $\E_{/X\times X\times Y}$. Since $(\id_{X\times X}, f\pr_1)$ is $L$-local (as the pullback of the $L$-local map $\Delta Y$), we further have $$
(\id_{X\times X}, f\pr_1)^{t\times Y}\simeq(\id_{X\times X}, f\pr_1)^{\Delta X\times Y}\simeq$$$$\simeq\prod_{\Delta X\times Y}(\Delta X\times Y)^\ast (\id_{X\times X}, f\pr_1)\simeq \prod_{\Delta X\times Y} (\id_X,f), $$
where $(\id_X,f)\colon X\rightarrow X\times Y$. We can now finally conclude because
$$
E'\simeq\sum_{X\times Y\rightarrow X}\left(\prod_{\pr_{X\times Y}\colon X\times X\times Y\rightarrow X\times Y}\left(\prod_{\Delta X\times Y} (\id_X,f)\right)\right)\simeq
$$$$\simeq \sum_{X\times Y\rightarrow X}\left(\prod_{\pr_{X\times Y}\circ (\Delta X\times Y)}(\id_X,f)\right)=\sum_{X\times Y\rightarrow X}(\id_X,f)=\id_X.$$
\end{proof}
\section{Existence of $L'$-localization}
\label{sec:existofl'loc}
We prove here that the class of $L$-separated maps is the class of local maps for a reflective subfibration on $\E$, and we start by proving a few preliminary results.

\smallskip  
Recall that, if $p,q$ are objects in a slice category $\E_{/Z}$, we write $p\times^{Z}q$ to mean the product object of $p$ and $q$ in $\E_{/Z}$.

\smallskip
The first result we need is a term-free interpretation of an internal Yoneda lemma involving diagonal maps.
\begin{lemma}
\label{lm:localyonlem}
Let $t\colon E\rightarrow X$ be a map in $\E$ and form the pullback square
$$
\bfig
\square|blrb|<500,250>[X\times E`E`X\times X`X;`X\times t`t`\pr_2]
\place(70,150)[\angle]
\efig
$$ 
Then there is a map in $\E_{/X^2}$
$$
\bfig
\Vtriangle|alr|<400,300>[E`X\times E`X\times X;(t,\id)`(\Delta X)t`X\times t]
\place(1000,200)[,]
\efig
$$
inducing an equivalence
$$
\beta\colon t\stackrel{\simeq}{\longrightarrow}\prod_{\pr_1}(X\times t)^{\Delta X}
$$
in $\E_{/X}$, where $\pr_1\colon X\times X\rightarrow X$ is the projection onto the first component.
\end{lemma}
\begin{proof}
For any $k\colon M\rightarrow X$, the product object $(k\times X)\times^{X^2}(\Delta X)$ in $\E_{/X^2}$ is given by $(\Delta X)k$. In fact, $(\Delta X)k$ is also the product object $(X\times k)\times^{X^2} (\Delta X)$ in $\E_{/X^2}$. Taking $k=t$, we get that $(t,\id)\colon (\Delta X)t\rightarrow X\times t$ gives a map $$\beta\colon t\longrightarrow\prod_{\pr_1}(X\times t)^{\Delta X}$$ by adjointness. Using the fact that $\Delta X$ is a section of $\pr_2$, and considering the adjoint pairs $\Sigma_{\pr_2}\dashv\pr_2^\ast$, $\pr_1^\ast\dashv\prod_{\pr_1}$, we get a chain of natural equivalences
$$
\E_{/X}(k,t)\simeq \E_{/X}(\pr_2(\Delta X)k,t)\simeq\E_{/X^2}\left( (\Delta X)k, X\times t \right)\simeq$$$$\simeq\E_{/X^2}\left( k\times X, (X\times t)^{\Delta X}\right)\simeq\E_{/X}\left(k,\prod_{\pr_1}(X\times t)^{\Delta X}\right)
$$
where the composite map is given by composition with $\beta$.
\end{proof}

\begin{lemma}
\label{lm:subsidlmforexistlprimeloc}
Let $X\in\E$ and let $r\colon R\rightarrow  X^2$ be an object in $\E_{/X^2}$. Let also $\ti{X\times r}$ be the composite map $(\tau\times X)\circ (X\times r)$, where $\tau\colon X^2\simeq X^2$ is the canonical involution.  Then the following hold.
\begin{itemize}
\item[(i)] There is a natural equivalence in $\E_{/X^2}$
$$
\beta\colon r\xrightarrow{\simeq}\prod_{\pr_{23}}(\ti{X\times r})^{(\Delta X\times  X)}$$
\item[(ii)] There is a map $\rho\colon \Delta X\rightarrow \prod_{\pr_{23}}(\ti{X\times r})^{(r\times  X)}$ such that, given any map $\eta\colon \Delta X\rightarrow~r$ in $\E_{/X^2}$, there is a commutative square
\begin{equation}
\label{eq:factoringsquare}
\bfig
\square/{-->}`{>}`{>}`{>}/<800,550>[\Delta X`\prod\limits_{\pr_{23}}(\ti{X\times r})^{(r\times X)}`r`\prod\limits_{\pr_{23}}(\ti{X\times r})^{(\Delta X\times X)};\rho`\eta`\prod\limits_{\pr_{23}}(\ti{X\times r})^{(\eta\times X)}`\beta]
\efig
\end{equation}
\end{itemize}
\end{lemma}
\begin{proof}
The first claim is a special case of \cref{lm:localyonlem} applied to the map $r=(r_1,r_2)\colon R\rightarrow X^2$, seen as a map $r\colon r_2\rightarrow \pr_2$ in $\E_{/X^2}$. Indeed, the following pullback square in $\E$
$$
\bfig
\square<500,250>[X^3`X^2`X^2`X;\pr_{13}`\pr_{23}`\pr_2`\pr_2]
\place(70,180)[\angle]
\efig
$$
witnesses that $\pr_3\colon X^3\rightarrow X$ is the product object of $\pr_2\colon X^2\rightarrow X$ with itself in $\E_{/X}$ and the displayed maps $\pr_{13}$ and  $\pr_{23}$ give the projection maps out of this product. The map $\Delta X\times X\colon X^2\rightarrow X^3$, seen as a map $\pr_3\rightarrow \pr_3$, is the diagonal of the object $\pr_3\in\E_{/X}$. Since $\ti{X\times r}=\pr_{13}^\ast(r)$, \cref{lm:localyonlem} gives the desired natural equivalence $\beta\colon r\simeq\prod_{\pr_{23}}(\ti{X\times r})^{(\Delta X\times X)}$.\par
For the second part, we describe the map $\rho$ and how it makes the square \eqref{eq:factoringsquare} commute by looking at its adjunct. Under the adjunction $\pr_{23}^\ast\dashv\prod_{\pr_{23}}$, giving a square as \eqref{eq:factoringsquare} is the same as giving a square 
$$
\bfig
\square/{-->}`{>}`{>}`{>}/<800,350>[X\times\Delta X`(\ti{X\times r})^{(r\times X)}`X\times r`(\ti{X\times r})^{(\Delta X\times X)};\rho'`X\times\eta`(\ti{X\times r})^{(\eta\times X)}`\beta']
\efig
$$
since $X\times\Delta X=\pr_{23}^\ast(\Delta X)$ and similarly for $X\times r$. Taking further adjoints along $(-)\times^{X^2}(\Delta X\times X)\dashv (-)^{\Delta X \times X}$, we need to exhibit a square
$$
\bfig
\square|alrb|/{>}`{>}`{>}`{-->}/<1900,350>[(X\times\Delta X)\times^{X^3}(\Delta X\times X)`(X\times r)\times^{X^3}(\Delta X\times X)`(X\times\Delta X)\times^{X^3}(r\times X)`\ti{X\times r};(X\times\eta)\times^{X^3}(\Delta X\times X)`(X\times \Delta X)\times^{X^3}(\eta\times X)`\beta^\sharp`\rho^\sharp]
\efig
$$
The products $(X\times\Delta X)\times^{X^3}(\Delta X\times X),\ (X\times r)\times^{X^3}(\Delta X\times X)$ and $(X\times\Delta X)\times^{X^3}(r\times X)$ in $\E_{/X^3}$, together with their projections onto the factors, are represented, in order, by the following pullback squares in $\E$
$$
\bfig
\node x(0,0)[X]
\node xx(600,0)[X^2]
\node xx1(0,-350)[X^2]
\node x3(600,-350)[X^3]
\node r(1200,0)[R]
\node xx2(1800,0)[X^2]
\node xr(1200,-350)[X\times R]
\node x31(1800,-350)[X^3]
\node r2(2300,0)[R]
\node rx(2900,0)[R\times X]
\node xx5(2300,-350)[X^2]
\node x32(2900,-350)[X^3]

\arrow[x`xx;\Delta X]
\arrow[x`xx1;\Delta X]
\arrow|m|[xx`x3;\Delta X\times X]
\arrow|b|[xx1`x3;X\times\Delta X]
\arrow|m|[x`x3;(\id,\id,\id)]
\arrow[r`xx2;r]
\arrow[r`xr;(r_1,\id)]
\arrow|m|[xx2`x31;\Delta X\times X]
\arrow|b|[xr`x31;X\times r]
\arrow|m|[r`x31;(r_1,r_1,r_2)]
\arrow[r2`rx;(\id,r_2)]
\arrow[r2`xx5;r]
\arrow|m|[rx`x32;r\times X]
\arrow|b|[xx5`x32;X\times\Delta X]
\arrow|m|[r2`x32;(r_1,r_2,r_2)]

\place(80,-80)[\angle]
\place(1280,-80)[\angle]
\place(2380,-80)[\angle]
\efig
$$
Using \cref{lm:localyonlem} as in the first part, we know the map $\beta^\sharp$ is given by 
$$
\bfig
\node r(0,0)[R]
\node xx(0,-300)[X^2]
\node xr(600,0)[X\times R]
\node xxx(600,-300)[X^3]

\arrow|a|[r`xr;\beta^\sharp=(r_1,\id)]
\arrow|l|[r`xx;r]
\arrow|r|[xr`xxx;\ti{X\times r}]
\arrow|b|[xx`xxx;\Delta X\times X]
\arrow|m|[r`xxx;(r_1,r_1,r_2)]
\efig
$$
We take $\rho^\sharp$ to be given by
$$
\bfig
\node r(0,0)[R]
\node xx(0,-300)[X^2]
\node xr(600,0)[X\times R]
\node xxx(600,-300)[X^3]

\arrow|a|[r`xr;\rho^\sharp=(r_2,\id)]
\arrow|l|[r`xx;r]
\arrow|r|[xr`xxx;\ti{X\times r}]
\arrow|b|[xx`xxx;X\times \Delta X]
\arrow|m|[r`xxx;(r_1,r_2,r_2)]
\efig
$$
Then the composite maps $\beta^\sharp\left((X\times\eta)\times^{X^3}(\Delta X\times X)\right)$ and $\rho^\sharp\left((X\times\Delta X)\times^{X^3}(\eta\times X)\right)$ are given by the following composite maps in $\E_{/X^3}$, respectively:
$$
\bfig
\node x(0,0)[X]
\node r(400,0)[R]
\node xr(900,0)[X\times R]
\node xxx(400,-350)[X^3]
\node x1(1500,0)[X]
\node r1(1900,0)[R]
\node xr1(2400,0)[X\times R]
\node xxx1(1900,-350)[X^3]

\arrow|a|[x`r;\eta]
\arrow|a|[r`xr;(r_1,\id)]
%\arrow||/@{>}@/_20pt/|(0.35){(\id,\id,\id)}/[x`xxx;]
\arrow|l|[x`xxx;(\id,\id,\id)]
\arrow|m|[r`xxx;(r_1,r_1,r_2)]
\arrow|r|[xr`xxx;\ti{X\times r}]
%\arrow||/@{>}@/^20pt/|(0.35){\ti{X\times r}}/[xr`xxx;]
\arrow|a|[x1`r1;\eta]
\arrow|a|[r1`xr1;(r_2,\id)]
\arrow|l|[x1`xxx1;(\id,\id,\id)]
\arrow|m|[r1`xxx1;(r_1,r_2,r_2)]
\arrow|r|[xr1`xxx1;\ti{X\times r}]
\efig
$$
By using properties of the product $X\times R$ and since $\eta$ is a section of both $r_1$ and $r_2$, one can see that these composite maps are equal since they are both equal to $(\id,\eta)\colon(\id,\id,\id)\rightarrow \ti{X\times r}$ in $\E_{/X^3}$. (The needed homotopies are obtained by using either degenerate $2$-simplices or the $2$-simplices defining $\eta\colon \Delta X\rightarrow r$.)
\end{proof}
\begin{theorem}[{\cite[Thm.~2.25]{locinhott}}]
\label{thm:existoflprimeloc}
For any $Y\in\E$, each $f\in\E_{/Y}$ has an $L'$-localization $\eta'_Y(f)\colon f\rightarrow f'$.
\end{theorem}
\begin{proof}
We prove the result for $Y=1$. Fix $X\in\E$ and let $\eta\colon\Delta X\rightarrow r$ be the $L$-reflection map of $\Delta X\in\E_{/X^2}$. Let $\kappa$ be a regular cardinal such that $r$ is relatively $\kappa$-compact and the relatively $\kappa$-compact $L$-local maps have a classifying map $u^\kappa_L\colon\ti{\U^\kappa_L}\rightarrow \U^\kappa_L$. Omitting $\kappa$ from our notation, we then have pullback squares
$$
\bfig
\square<400,350>[R`\ti{\U_L}`X\times X`\U_L;\ti{r}`r`u_L`\li{r}]
\place(70,280)[\angle]
\place(900,125)[\text{and}]
\square(1500,0)/{>}`{>}`{>}`{ >->}/<400,350>[\ti{\U_L}`\ti{\U}`\U_L`\U;`u_L`u`\iota]
\place(1570,280)[\angle]
\efig
$$ 
We denote the composite pullback square as
\begin{equation}
\label{eq:namersquare}
\bfig
\square<400,350>[R`\ti{\U}`X\times X`\U;\ti{r}`r`u`\name{r}]
\place(70,260)[\angle]
\efig 
\end{equation}
Let $(\eta',i)$ be the (effective epi,mono)-factorization of $\name{r}^\sharp\colon X\rightarrow\U^X$, the adjunct map to $\name{r}$. Set $X':=\cod(\eta')$. Note that, if $(\eta'_L,i_L)$ is the (effective epi,mono)-factorization of $\li{r}^\sharp$, then $\eta'=\eta'_L$ and $i=\iota^X\circ i_L$ since $\iota^X$ is a mono.

\smallskip
Our goal is to apply \cref{thm:charoflprimeloc} to $\eta'$. The map $\eta'$ is an effective epi by definition. To show that $X'$ is $L$-separated, note first that $\U_L$ is $L$-separated  by \cref{lm:ULisLsep}, hence so is $\U_L^X$, by \cref{prop:closureoflsepmapsunderpllbckdepprod} (1). Since $i$ is a mono, we have that $\Delta X'=(i_L\times i_L)^\ast (\Delta (\U_L^X))$, which implies that $X'$ is $L$-separated, because $L$-local maps are closed under pullbacks. It remains to show that $\Delta \eta'$ is the $L$-localization map of $\Delta X$. We can see $\Delta\eta'$ as a map $\Delta\eta'\colon\Delta X\rightarrow t$ in $\E_{/X^2}$, where $t$ is the pullback map $(\eta'\times\eta')^\ast(\Delta X')$ and it is therefore $L$-local. Hence, there is a unique map $\phi\colon r\rightarrow t$ with $\phi\eta=\Delta\eta'$ as maps in $\E_{/X^2}$. We will show that $\phi$ is an equivalence.

\smallskip
The strategy we adopt is to, first, construct a monomorphism $\phi'\colon t\rightarrowtail r$ and, then, show that $\phi'\phi\colon r\rightarrow r$ is an equivalence by showing that we have $\phi'\phi\eta=\eta$. This will imply that $\phi$ itself is an equivalence. Note that, by definition of $\phi$, showing that $\phi'\phi\eta=\eta$ is the same as showing that $\phi'\Delta\eta'=\eta$.
\smallskip
\paragraph{\textbf{Step 1. Construction of $\phi'$ and description of $\phi'\Delta\eta'$}} We construct $\phi'$ as a composite of some equivalences and a monomorphism. Consider the diagram:\begin{equation}
\label{eq:auxiliarydiagramforl'loc}
\tag{D}
\bfig
\node x(0,0)[X]
\node ux(800,0)[\U^X]
\node xux(1600,0)[X\times\U^X]
\node u(2250,0)[\U]
\node eq(2620,0)[\Eq_{\U}(\ti{\U})]
\node map(3100,-330)[M]
\node xx(0,-400)[X^2]
\node uxux(800,-400)[\U^X\times\U^X]
\node xuxux(1600,-400)[X\times\U^X\times\U^X]
\node uu(2250,-400)[\U\times\U]
\node xxx(800,-1200)[X^3]
\node w(2000,-1150)[W]
\node xx1(800,-800)[X^2]

\arrow[x`ux;\name{r}^\sharp]
\arrow[xux`ux;\pr_2]
\arrow[x`xx;\Delta X]
\arrow|r|[ux`uxux;\Delta\U^X]
\arrow|m|[xux`xuxux;\ X\times\Delta\U^X]
\arrow|b|/@{->}|!{(800,-800);(0,0)}\hole_<>(.53){\name{r}^\sharp\times\name{r}^\sharp}/[xx`uxux;]
\arrow||/@{->}|!{(800,-800);(1600,0)}\hole_<>(.53){\pr_{23}}/[xuxux`uxux;]
\arrow||/@{->}|(0.3){X\times\Delta X}/[xx1`xxx;]
\arrow||/@{->}|(0.65){\pr_2}/[xx1`x;]
\arrow|m|[xxx`xx;\pr_{23}]
\arrow||/@{->}|(0.3){X\times\name{r}^\sharp}/[xx1`xux;]
\arrow||/@{->}|!{(800,-800);(2000,-1150)}\hole|(0.6){X\times\name{r}^\sharp\times\name{r}^\sharp}/[xxx`xuxux;]
\arrow[xuxux`uu;\ev]
\arrow|a|[xux`u;\ev]
\arrow|l|[u`uu;\Delta\U]
\arrow|b|[u`eq;\simeq]
\arrow/@{->}|(0.4){\Eq_{\U}(u)}/[eq`uu;]
\arrow|m|[map`uu;(\id\times u)^{(u\times \id)}]
\arrow|a|/{@{ >->}@/^35pt/}/[u`map;j]
\arrow||/@{->}|!{(800,-800);(2000,-1150)}\hole|(0.75){(\name{r}\pr_{12},\name{r}\pr_{13})}/[xxx`uu;]
\arrow||/{@{ >->}@/^5pt/}/[eq`map;]
\arrow||/@{-->}^(0.65){\psi}/[xx1`w;]
\arrow|b|[w`xxx;(\ti{X\times r})^{(r\times X)}]
\arrow|r|[w`map;\sigma]

\place(400,-200)[\text{\tiny (1)}]
\place(1200,-200)[\text{\tiny (2)}]
\place(1950,-200)[\text{\tiny (3)}]
\place(700,-600)[\text{\tiny (4)}]
\place(700,-1000)[\text{\tiny (5)}]
\place(1300,-550)[\text{\tiny (6)}]
\place(1870,-1000)[\text{\tiny (7)}]

\place(1350,-20)[\begin{rotate}{-90}\angle\end{rotate}]
\place(850,-650)[\begin{rotate}{140}\angle\end{rotate}]
\place(700,-740)[\begin{rotate}{230}\angle\end{rotate}]
\place(2050,-1030)[\begin{rotate}{140}\angle\end{rotate}]
\place(850,-1050)[\begin{rotate}{140}\angle\end{rotate}]
\efig
\end{equation}
The maps labelled as $\ev$ are appropriate counits of product $\dashv$ internal-hom adjunctions. We proceed to explain this diagram, show how it defines $\phi'$, and give a description of $\phi'\Delta\eta'$.

\smallskip
\noindent (i) Recall that $\Delta\eta'$ is a map $\Delta X\rightarrow t$ in $\E_{/X^2}$, and one can show that $t$ is the pullback map of the cospan in (1) of \eqref{eq:auxiliarydiagramforl'loc}. Because of this, the square (1) determines $\Delta\eta'$.

\smallskip
\noindent (ii) Thanks to Function Extensionality (\cref{prop:funext}), $\Delta\U^X\simeq\Pi_{\pr_{23}}\ev^{\ast}(\Delta\U)$. Hence,
$t\simeq (\name{r}^\sharp \times\name{r}^\sharp)^\ast\left(\Pi_{\pr_{23}}\ev^{\ast}(\Delta\U)\right).$

\smallskip
\noindent (iii) Since the bottom square (5) in \eqref{eq:auxiliarydiagramforl'loc} is a pullback, we can use the Beck-Chevalley condition (\cref{lm:BCcond}) to get an equivalence
$t\simeq \Pi_{\pr_{23}}(\name{r}\pr_{12},\name{r}\pr_{13})^\ast(\Delta\U)$.
Since the pullback of $X\times\Delta \U^X$ along $X\times\name{r}^\sharp\times\name{r}^\sharp$ is $X\times t$, the square (6) in \eqref{eq:auxiliarydiagramforl'loc} determines the map $X\times\Delta\eta'\colon X\times\Delta X\rightarrow X\times t$ in $\E_{/X^3}$. It follows that the map $X\times\Delta X\rightarrow (\name{r}\pr_{12},\name{r}\pr_{13})^\ast(\Delta\U)$ determined by the square given as the composite of (3) and (6) is the adjunct of the composite map
$$
\Delta X\xrightarrow{\Delta\eta'}t\simeq \prod_{\pr_{23}}(\name{r}\pr_{12},\name{r}\pr_{13})^\ast(\Delta\U)$$
\smallskip
\noindent (iv) We now consider the map $j$ in $\E_{/\U^2}$ displayed in the top-right corner of \eqref{eq:auxiliarydiagramforl'loc}. Here, $M$ is simply a name for the domain of the map $(\id\times u)^{(u\times\id)}$. The map $j$ is defined as the composite of the equivalence $\Delta\U\simeq\Eq_{\U}(u)$, given by univalence, and the monomorphism $\Eq_{\U}(u)\rightarrowtail (\id\times u)^{(u\times \id)}$. Thus, $j$ is a mono as well. Using the fact that (7) in \eqref{eq:auxiliarydiagramforl'loc} is a pullback square, we obtain a monomorphism
$$\prod_{\pr_{23}}(\name{r}\pr_{12},\name{r}\pr_{13})^\ast(\Delta\U)\xrightarrowtail{\prod\limits_{\pr_{23}}(\name{r}\pr_{12},\name{r}\pr_{13})^\ast(j)} \prod_{\pr_{23}}(\ti{X\times r})^{(r\times X)}.$$ 
Here, $\ti{X\times r}$ is the pullback map $\pr_{13}^\ast(r)=(\tau\times X)(X\times r)$, where $\tau\colon X^2\simeq X^2$ is the swapping equivalence, and $W$ is simply a name for the domain of the map $(\ti{X\times r})^{(r\times X)}$. Note that the map displayed above is indeed a monomorphism because, being right adjoints, pullback and dependent-product functors preserve monomorphisms. Therefore, we get a composite monomorphism
$$t\xrightarrowtail{\ \ } \prod_{\pr_{23}}(\ti{X\times r})^{(r\times X)}.$$
The map $\psi$ in $\E_{/X^3}$ given in \eqref{eq:auxiliarydiagramforl'loc} is determined, as a map $X\times\Delta X\rightarrow (\ti{X\times r})^{(r\times X)}$, by the composite of the squares (3) and (6) with the 2-simplex representing the map $j\colon \Delta(\U)\rightarrowtail (\id\times u)^{(u\times\id)}$. It follows that $\psi$ is the adjunct to the composite
$$
\Delta X\xrightarrow{\Delta\eta'}t\xrightarrowtail{\ \ } \prod_{\pr_{23}}(\ti{X\times r})^{(r\times X)}
$$
This means that this latter map is the composite
$$
\Delta X\xrightarrow{\gamma}\prod_{\pr_{23}} X\times\Delta X\xrightarrow{\prod\limits_{\pr_{23}}\psi}\prod_{\pr_{23}}(\ti{X\times r})^{(r\times X)},
$$
where $\gamma$ is the unit of the adjunction $\pr_{23}^\ast\dashv \prod_{\pr_{23}}$ at $\Delta X$.

\smallskip
\noindent (v) Since $\ti{X\times r}=\pr_{13}^\ast(r)$, $\ti{X\times r}$ is $L$-local. Hence, because $\eta\times X\colon\Delta X\times X\rightarrow r\times X$ is an $L$-localization map (it is the pullback along $\pr_{12}$ of $\eta$), we have an equivalence $$(\ti{X\times r})^{(\eta\times X)}\colon(\ti{X\times r})^{(r\times X)}\xrightarrow{\simeq}(\ti{X\times r})^{(\Delta X\times X)}.$$
Whence, we have a composite monomorphism
$$t\xrightarrowtail{\ \ } \prod_{\pr_{23}}(\ti{X\times r})^{(r\times X)}\simeq\prod_{\pr_{23}}(\ti{X\times r})^{(\Delta X\times X)}.$$
\smallskip
\noindent (vi) Finally, we have an equivalence $\beta\colon r\xrightarrow{\simeq}\prod_{\pr_{23}}(\ti{X\times r})^{(\Delta X\times X)}$
as in \cref{lm:subsidlmforexistlprimeloc}. Composing the monomorphism obtained in (v) with the inverse of $\beta$ we obtain the needed monomorphism $\phi'\colon t\rightarrowtail r$. Using what we found in (iv) above, the composite $\phi'\Delta\eta'$ is then given as the composite
$$
\Delta X\xrightarrow{\gamma}\prod_{\pr_{23}} X\times\Delta X\xrightarrow{\prod\limits_{\pr_{23}}\psi}\prod_{\pr_{23}}(\ti{X\times r})^{(r\times X)}\xrightarrow{\simeq}r,
$$
where the displayed equivalence is $\beta^{-1}\prod\limits_{\pr_{23}}(\ti{X\times r})^{(\eta\times X)}$.
\paragraph{\textbf{Step 2. Proof that $\phi'\Delta\eta'=\eta$.}} By the work above, it suffices to show that the maps
$$
\Delta X\xrightarrow{\eta}r\xrightarrow[\ \ \simeq\ \ ]{\ \ \beta\ \ }\prod_{\pr_{23}}(\ti{X\times r})^{(\Delta X\times X)}
$$
and
$$
\Delta X\xrightarrow{\gamma}\prod_{\pr_{23}}X\times \Delta X\xrightarrow{\prod\limits_{\pr_{23}}\psi}\prod_{\pr_{23}}(\ti{X\times r})^{(r\times X)}\xrightarrow[\simeq]{\prod\limits_{\pr_{23}}(\ti{X\times r})^{(\eta\times X)}}\prod_{\pr_{23}}(\ti{X\times r})^{(\Delta X\times X)}
$$
are equal in $\E_{/X^2}$. By \cref{lm:subsidlmforexistlprimeloc} (ii), there is a map $\rho\colon\Delta X\rightarrow \prod\limits_{\pr_{23}}(\ti{X\times r})^{(r\times X)}$ making the following diagram commute in $\E_{/X^2}$
$$
\bfig
\square/{-->}`{>}`{>}`{>}/<800,500>[\Delta X`\prod\limits_{\pr_{23}}(\ti{X\times r})^{(r\times X)}`r`\prod\limits_{\pr_{23}}(\ti{X\times r})^{(\Delta X\times X)};\rho`\eta`\prod\limits_{\pr_{23}}(\ti{X\times r})^{(\eta\times X)}`\beta]
\efig
$$
Thus, we only need to show that $\rho=\left(\prod_{\pr_{23}}\psi\right)\gamma$. Equivalently, we can show that the adjunct maps $\rho ',\psi\colon (X\times\Delta X)\rightarrow (\ti{X\times r})^{(r\times X)}$ are equal in $\E_{/X^3}$. Since the square (7) in the diagram (\ref{eq:auxiliarydiagramforl'loc}) is a pullback, we only need to show that $\rho'$ and $\psi$ are equal after composing with $g:=(\name{r}\pr_{12},\name{r}\pr_{13})$ and $$\sigma\colon g (\ti{X\times r})^{(r\times X)}\rightarrow (\id\times u)^{(u\times\id)},$$ that is, as maps $g(X\times \Delta X)\rightarrow (\id\times u)^{(u\times\id)}$. Finally, we can further show that $\sigma\rho',\ \sigma\psi$ are equal in $\E_{/\U^2}$ by showing their adjuncts along the adjunction $(-)\times^{\U^2}(u\times\id)\dashv (-)^{(u\times\id)}$ are equal.

\smallskip
In order to describe the adjunct of $\sigma\rho'$, we use \cref{lm:adjointopllbckofhommaps} with $f=X\times\Delta X$, $g:=(\name{r}\pr_{12},\name{r}\pr_{13})$, $p=\id\times u$ and $q=u\times\id$. Consequently, $g^\ast q=r\times X$, $g^\ast p=\ti{X\times r}$ and the adjunct of $\sigma\rho'$ is given as the composite map $$g((X\times\Delta X)\times^{X^3} (r\times X))\xrightarrow{\rho^\sharp} g(\ti{X\times r}) =gg^\ast (\id\times u)\xrightarrow{\epsilon_{(\id\times u)}}\id\times u$$
Recall the pullback square \eqref{eq:namersquare} defining $\name{r}$. Since $(X\times\Delta X)\times^{X^3} (r\times X)=(X\times\Delta X)r$, using the proof of \cref{lm:subsidlmforexistlprimeloc} and the fact that $g^\ast (\id\times u)=\ti{X\times r}$, we have that $\rho^\sharp\colon(X\times\Delta X)r\rightarrow \ti{X\times r}$ and $\epsilon_{(\id\times u)}\colon g(\ti{X\times r})\rightarrow \id\times u$ are described by the two squares below
$$
\bfig
\node r(0,0)[R]
\node xx(0,-300)[X^2]
\node xr(600,0)[X\times R]
\node xxx(600,-300)[X^3]

\node xr1(1300,0)[X\times R]
\node xxx1(1300,-300)[X^3]
\node uut(2700,0)[\U\times\ti{\U}]
\node uu(2700,-300)[\U^2]

\arrow|a|[r`xr;\rho^\sharp=(r_2,\id)]
\arrow|l|[r`xx;r]
\arrow|r|[xr`xxx;\ti{X\times r}]
\arrow|b|[xx`xxx;X\times \Delta X]
\arrow|m|[r`xxx;(r_1,r_2,r_2)]

\arrow|a|[xr1`uut;\epsilon_{(\id\times u)}=(\name{r}(r_1\pr_2,\pr_1),\ti{r}\pr_2)]
\arrow|l|[xr1`xxx1;\ti{X\times r}]
\arrow|r|[uut`uu;\id\times u]
\arrow|b|[xxx1`uu;g=(\name{r}\pr_{12},\name{r}\pr_{13})]
%\arrow|m|[xr1`uu;(\name{r}(r_1\pr_2,\pr_1),\name{r}r\pr_2)]

\place(1370,-90)[\angle]
\efig
$$
Hence, the composite $\epsilon_{(\id\times u)}\rho^\sharp$ (the adjunct of $\sigma \rho'$ in $\E_{/\U^2}$) is given by the map
\begin{equation}
\label{eq:adjtosigrho}
\bfig
\node r(0,0)[R]
\node uut(700,0)[\U\times\ti{\U}]
\node uu(350,-300)[\U^2]

\arrow|a|[r`uut;(\name{r}r,\ti{r})]
\arrow|l|[r`uu;(\name{r},\name{r})r]
\arrow|r|[uut`uu;\id\times u]
\efig
\end{equation}
To describe the adjunct of $\sigma \psi$, note that, from the squares (6), (7) and (3) and the definition of $j$ in the diagram \eqref{eq:auxiliarydiagramforl'loc}, $\sigma\psi$ is given as the map in $\E_{/\U^2}$ described by the diagram
$$
\bfig
\node xx(0,0)[X^2]
\node u(900,0)[\U]
\node xxx(0,-350)[X^3]
\node uu(900,-350)[\U^2]
\node m(1500,0)[M]

\arrow[xx`u;\name{r}]
\arrow[u`m;j]
\arrow[xx`xxx;X\times\Delta X]
\arrow|m|[u`uu;\Delta \U]
\arrow|b|[xxx`uu;g=(\name{r}\pr_{12},\name{r}\pr_{13})]
\arrow|m|[xx`uu;(\name{r},\name{r})]
\arrow|r|[m`uu;(\id\times u)^{(u\times\id)}]
\efig
$$
Then, the adjunct of $j\name{r}$ in $\E_{/\U^2}$ is the composite $j^\sharp (\name{r}\times^{\U^2}(u\times\id))$, where $j^\sharp\colon\Delta \U\times^{\U^2}(u\times\id)\rightarrow \id\times u$ is the adjunct of $j$. Using that there are pullback squares
$$
\bfig
\node r(0,0)[R]
\node utu(600,0)[\ti{\U}\times\U]
\node xx(0,-300)[X^2]
\node uu(600,-300)[\U^2]

\node ut(1200,0)[\ti{\U}]
\node utu1(1700,0)[\ti{\U}\times\U]
\node u(1200,-300)[\U]
\node uu1(1700,-300)[\U^2]

\arrow|a|[r`utu;(\ti{r},\name{r}r)]
\arrow|l|[r`xx;r]
\arrow|r|[utu`uu;u\times\id]
\arrow|b|[xx`uu;(\name{r},\name{r})]

\arrow|a|[ut`utu1;(\id,u)]
\arrow|l|[ut`u;u]
\arrow|r|[utu1`uu1;u\times\id]
\arrow|b|[u`uu1;\Delta\U]

\place(70,-70)[\angle]
\place(1270,-70)[\angle]
\efig
$$ 
we get that $j^\sharp (\name{r}\times^{\U^2}(u\times\id))$ is the composite map
\begin{equation}
\label{eq:adjtosigpsi}
\bfig
\node r(0,0)[R]
\node ut(550,0)[\ti{\U}]
\node uut(1150,0)[\U\times\ti{\U}]
\node uu(550,-320)[\U^2]

\node r1(2100,0)[R]
\node uut1(2700,0)[\U\times\ti{\U}]
\node uu1(2400,-320)[\U^2]

\arrow|a|[r`ut;\ti{r}]
\arrow|a|[ut`uut;j^\sharp=(u,\id)]
\arrow|l|[r`uu;(\name{r},\name{r})r]
\arrow|m|[ut`uu;(u,u)]
\arrow|r|[uut`uu;\id\times u]

\arrow|a|[r1`uut1;(u\ti{r},\ti{r})]
\arrow|l|[r1`uu1;(\name{r},\name{r})r]
\arrow|r|[uut1`uu1;\id\times u]

\place(1500,-175)[=]
\efig
\end{equation}
One can now see that the maps \eqref{eq:adjtosigrho} and \eqref{eq:adjtosigpsi} are equal by using the square \eqref{eq:namersquare} defining $\name{r}$ (including the implicit given homotopies). Our proof is then complete.
\end{proof}
Once we know that every map in $\E$ has an $L'$-localization, we can also show that $L'$-localization form a reflective subfibration on $\E$. The crucial point here is to show pullback-compatibility of $L'$-reflections. This is necessary when working in higher topos theory, but it is superfluous in homotopy type theory as reflections are automatically stable under pullbacks in that setting.
\begin{corollary}
\label{cor:l'reflsubf}
Given any reflective subfibration $L_\bullet$ of an $\infty$-topos $\E$, there exists a reflective subfibration $L'_\bullet$ of $\E$ such that the $L'$-local maps are exactly the $L$-separated maps. Furthermore, if $L_\bullet$ is a modality, then so is $L'_\bullet$.
\end{corollary}
\begin{proof}
Let $\D'$ be the full subcategory of $\E$ spanned by the $L$-separated objects and let $\iota\colon\D'\rightarrow\E$ be the inclusion functor. \cref{thm:existoflprimeloc} constructs, for every $X\in\E$, an $L'$-localization map $\eta(X)\colon X\rightarrow L'(X)$. By definition of $L'$-localization map, this means that, for every $X\in\E$, the $\infty$-category defined as the pullback 
$$
\bfig
\node xd(0,0)[X_{/\D'}]
\node d(400,0)[\D']
\node xe(0,-250)[X_{/\E}]
\node e(400,-250)[\E]

\arrow[xd`d;]
\arrow[xe`e;]
\arrow[xd`xe;]
\arrow|r|[d`e;\iota]

\place(80,-80)[\angle]
\efig
$$
has an initial object. By \cite[\S 17.4]{notesonqcat}, $\iota$ has a left adjoint $L'\colon\E\rightarrow\D'$, i.e., $\D'$ is a reflective subcategory of $\E$. The same construction performed on each slice category now gives that, for every $X\in\E$, the full subcategory $\D'_X$ of $\E_{/X}$ on the $L$-separated $p\in\E_{/X}$ is reflective. Since $L$-separated maps are closed under pullbacks (see \cref{prop:closureoflsepmapsunderpllbckdepprod}), we obtain a system of reflective subcategories $L'_\bullet$ on $\E$. To conclude that we actually get a reflective subfibration, we only need to verify that the $L'$-reflection maps are compatible with pullbacks.\par
Let then $p\colon E\rightarrow X$ be an object in $\E_{/X}$ and $f\colon Y\rightarrow X$ a map in $\E$. Let 
$$
\bfig
\Vtriangle|alr|<300,250>[E`E'`X;\eta':=\eta'_X(p)`p`p']
\efig
$$
be the $L'$-localization of $p$. We need to show that $m:= f^\ast (\eta')\colon f^\ast (p)\rightarrow f^\ast (p')$ is the $L'$-localization of $f^\ast (p)$ in $\E_{/Y}$. To do so we use \cref{thm:charoflprimeloc}. Set $f^\ast(E):=Y\times_X E$, $q:=f^\ast (p)$ and $f^\ast(E'):=Y\times_X E'$. Since $\eta'$ is an effective epimorphism and effective epimorphisms are closed under pullbacks, an application of the pasting lemma for pullbacks show that $m$ is also an effective epimorphism. By \cref{prop:closureoflsepmapsunderpllbckdepprod} (1), $f^\ast (p')$ is $L$-separated. Therefore, we only need to show that $\Delta (m)$, as a map in $\E_{/f^\ast(E)\times_Y f^\ast(E)}$, is the $L'$-localization map of $\Delta q$. In $\E_{/Y}$ we have the pullback square (products are products in $\E_{/Y}$)
$$
\bfig
\square<800,350>[q\times_{f^\ast (p')}q`f^\ast (p')`q\times q`f^\ast (p')\times f^\ast (p');`(m\times m)^\ast (\Delta(f^\ast (p'))`\Delta(f^\ast (p'))`m\times m]
\place(100,230)[\angle]
\efig
$$
and $\Delta m$ is a map $\Delta q\rightarrow (m\times m)^\ast (\Delta(f^\ast (p'))$ in $\left(\E_{/Y}\right)_{/(q\times q)}$. Since $\left(\E_{/Y}\right)_{/(q\times q)}\simeq\E_{/f^\ast(E)\times_Y f^\ast(E)}$ and $m=f^\ast (\eta')$, one can see that $\Delta m$ is the map
$$
\bfig
\Vtriangle|amm|<500,350>[f^\ast(E)`f^\ast(E)\times_{f^{\ast}(E')}f^\ast(E)`f^\ast(E)\times_Y f^\ast(E);\Delta m`\Delta q`t]
\efig
$$
in $\E_{/f^\ast(E)\times_Y f^\ast(E)}$, where $t$ corresponds to the map $(m\times m)^\ast (\Delta(f^\ast (p'))$ above. Similarly, $\Delta\eta'$ is a map $\Delta p\rightarrow s$ in $\E_{/E\times_X E}$ (where $s$ is a suitable pull-backed map) and it is the $L$-localization of $\Delta p$ by \cref{thm:charoflprimeloc}. We want to show that $\Delta m$ is a pullback of this $L$-localization and conclude because $L_\bullet$ is a reflective subfibration.
Let $g\colon f^\ast(E)\rightarrow E$ and $g'\colon f^\ast(E')\rightarrow E'$ be the projection maps.
As in the proof of \cref{prop:closureoflsepmapsunderpllbckdepprod}, we see that the following are all pullback squares in $\E$
$$
\bfig
\square(-600,0)<800,250>[f^\ast(E)\times_Y f^\ast(E)`E\times_X E`f^\ast(E)`E;```g]
\square(1000,0)<800,250>[f^\ast(E')\times_Y f^\ast(E')`E'\times_X E'`f^\ast(E')`E';```g']
\place(-500,150)[\angle]
\place(1100,150)[\angle]
\efig
$$
$$
\bfig
\square(-600,0)<800,300>[f^\ast(E)`E`f^\ast(E)\times_Y f^\ast(E)`E\times_X E;g`\Delta(f^\ast(p))`\Delta p`]
\square(1000,0)<800,300>[f^\ast(E')`E'`f^\ast(E')\times_Y f^\ast(E')`E'\times_X E';g'`\Delta(f^\ast(p'))`\Delta p'`]
\place(-500,200)[\angle]
\place(1100,200)[\angle]
\efig
$$
Then in the diagram
$$
\bfig
\cube|alll|<900,800>[f^\ast(E)\times_{f^\ast(E')}f^\ast(E)`E\times_{E'}E`f^\ast(E)\times_Y f^\ast(E)`E\times_{X}E;`t`s`]%
(600,-300)|bmrb|<900,800>[f^\ast(E')`E'`f^\ast(E')\times_Y f^\ast(E')`E'\times_X E';g'`\Delta (f^\ast(p'))`\Delta p'`]%
[``m\times_Y m`\eta'\times_X\eta']
\efig
$$
the left and right sides are pullbacks (by definition of $t$ and $s$) and the front square is a pullback by the above. Therefore, the back square is also a pullback. A final application of the pasting lemma now shows that there are pullback squares in $\E$
$$
\bfig
\hSquares(0,0)%
<300>[f^\ast(E)`f^\ast(E)\times_{f^\ast(E')} f^\ast(E)`f^\ast(E)\times_{Y} f^\ast(E)`E`E\times_{E'}E`E\times_X E;\Delta m`t`g```\Delta \eta'`s]
\place(100,180)[\angle]
\place(1200,180)[\angle]
\efig
$$
completing the proof that $L'_\bullet$ is a reflective subfibration.\par
The final claim about $L'$ being a modality when $L$ is follows from the observation that, given composable maps $f\colon X\rightarrow Y$ and $g\colon Y\rightarrow Z$ in $\E$, we have $\Delta(gf)=p\Delta f$, where $p$ is the leftmost vertical map in 
$$
\bfig
\hSquares(0,0)%
<300>[X\times_Y X`X`Y`X\times_Z X`X\times_Z Y`Y\times_Z Y;`f```\Delta g`\id_X\times_Z f`f\times_Z\id_Y]
\place(70,200)[\angle]
\place(1080,200)[\angle]
\efig
$$ 
Therefore, if $g$ is $L$-separated (so that $\Delta g$ is $L$-local), $p$ is $L$-local. If also $f$ is $L$-separated and $L$ is a modality, we can then conclude from $\Delta(gf)=p\Delta f$ that $gf$ is $L$-separated.
\end{proof}
\section{Appendix: On locally cartesian closed $\infty$-categories}\label{AppA}
%\myappendices{Appendix \ref{AppA} \byname{AppA}}
We prove here some miscellaneous facts about locally cartesian closed (lcc) $\infty$-categories that we need but we could not fit elsewhere. Some of these results are well-known, but others do not seem to appear or be proven in the literature.\par
In \cref{sec:pullbackfunctandadj}, we discuss some results about cartesian-closedness of pullback functors, and some interactions between their adjoints. In \cref{sec:funext}, we give a ``term-free" version of the type-theoretic axiom known as \emph{function extensionality}, and we prove that it holds in any lcc $\infty$-category. Finally, in \cref{subsec:contractibility}, we prove a ``fiberwise" criterion for extending a map along another one with the same domain.

\smallskip
We fix throughout an lcc $\infty$-category $\C$. 

\subsection{Pullback functor and its adjoints}\label{sec:pullbackfunctandadj} The first set of results we need explore the behaviours of the pullback functors and of their adjoints in $\C$.
\begin{lemma}
\label{lm:pllbckcartclosedfunct}
Let $\C$ be a locally cartesian closed $\infty$-category. For any morphism $g\colon Y\rightarrow X$ in $\C$ the pullback functor
$
g^\ast\colon\C_{/X}\rightarrow\C_{/Y}
$
is cartesian closed, i.e., for every $p,q\in\C_{/X}$, $g^\ast\left(p^q\right)$ is the exponential object $g^\ast (p)^{g^\ast (q)}$ in $\C_{/Y}$.
\end{lemma}
A proof of the above result for $1$-categories can be found in \cite[Lemma A.1.5.2]{sketches} and the same proof carries over to $\infty$-categories.
\begin{lemma}
\label{lm:adjointopllbckofhommaps}
Let $\epsilon\colon gg^\ast\rightarrow \id_{\C_{/X}}$ be the counit of the adjunction $g\circ (-)\dashv g^\ast$. Given $X\in\C$, take $p,q\in\C_{/X}$. Suppose given a diagram  in $\C$
$$
\bfig
\node a(0,0)[A]
\node w(400,-100)[W]
\node t(1000,-100)[T]
\node y(400,-450)[Y]
\node x(1000,-450)[X]

\arrow|a|[a`w;\rho]
\arrow|a|[w`t;\sigma:=\epsilon_{p^q}]
\arrow|l|/{@{>}@/_1em/}/[a`y;f]
\arrow|m|[w`y;(g^\ast p)^{(g^\ast q)}]
\arrow|r|[t`x;p^q]
\arrow|b|[y`x;g]

\place(470,-170)[\angle]
\efig
$$
Let $\rho^\sharp\colon f\times^Y g^\ast q\rightarrow g^\ast p$ be the adjunct to $\rho$ in $\C_{/Y}$ and consider the map $\sigma \rho\colon gf\rightarrow ~p^q$ in $\C_{/X}$. Then, $g(f\times^Y g^\ast q)=gf\times^X q$ and the adjunct of $\sigma \rho$ is given by the composite map:
$
g(f\times^Y g^\ast q)\xrightarrow{\rho^\sharp} gg^\ast p\xrightarrow{\epsilon_p}p.
$
\end{lemma}
\begin{proof}
The fact that $g(f\times^Y g^\ast q)=gf\times^X q$ is given by the pasting-lemma for pullbacks. By definition, the adjunct of $\sigma \rho$ is the composite $$gf\times^X q\xrightarrow{\sigma \rho\times^X q}p^q\times^X q\xrightarrow{\ev_{p,q}}p$$ and the adjunct $\rho^\sharp$ is the composite $$f\times^Y g^\ast q\xrightarrow{\rho\times^Y g^\ast q}(g^\ast p)^{(g^\ast q)}\times^Y g^\ast q\xrightarrow{\ev_{g^\ast p,g^\ast q}}g^\ast p.$$ Using that $(g^\ast p)^{(g^\ast q)}\times^Y g^\ast q= g^\ast (p^q\times^X q)$, the map $\ev_{g^\ast p,g^\ast q}$ is the map $g^\ast (\ev_{p,q})$. One then needs to show that the maps $\ev_{p,q}(\sigma\rho\times^X q)$ and $\epsilon_p g^\ast(\ev_{p,q})(\rho\times^Y g^\ast q)$ are equal. Consider the diagram below, where all squares are pullbacks
$$
\bfig
\node a(-700,-300)[A]
\node w(0,0)[W]
\node e(800,0)[T]
\node y(300,-300)[Y]
\node x(1100,-300)[X]
\node gfq(-700,300)[(gf)^\ast Q]
\node geq(0,600)[g^\ast(T\times_X Q)]
\node gq(300,300)[g^\ast Q]
\node eq(800,600)[T\times_X Q]
\node q(1100,300)[Q]

\arrow|b|[a`y;f]
\arrow|b|[y`x;g]
\arrow[gfq`gq;]
\arrow[gq`q;]
\arrow[geq`eq;\sigma']
\arrow[gfq`geq;m]
\arrow[eq`q;]
\arrow[geq`gq;]
\arrow[w`y;]
\arrow|m|[e`x;p^q]
\arrow|l|[gfq`a;(gf)^\ast q]
\arrow||/@{->}|!{(0,0);(800,0)}\hole|(0.6){g^\ast q}/[gq`y;]
\arrow|r|[q`x;q]
\arrow|a|[a`w;\rho]
\arrow||/@{->}|!{(300,300);(300,-300)}\hole|(0.6){\sigma}/[w`e;]
\arrow||/@{->}|!{(-700,300);(300,300)}\hole|(0.5){}/[geq`w;]
\arrow||/@{->}|!{(300,300);(1100,300)}\hole|(0.5){}/[eq`e;]
\efig
$$
Then $m$ (as a map over $Y$) is $\rho\times^Y g^\ast q$ and $\sigma'm$ (as a map over $X$) is $\sigma\rho\times^X q$. The claim now follows thanks to the following commutative diagram, where the back, front and bottom faces of the cube (and, hence, also the top face) are pullbacks
$$
\bfig
\node a(-700,-300)[A]
\node w(0,0)[W]
\node e(800,0)[T]
\node y(300,-300)[Y]
\node x(1100,-300)[X]
\node gfq(-700,300)[(gf)^\ast Q]
\node geq(0,600)[g^\ast(T\times_X Q)]
\node gp(300,300)[g^\ast P]
\node eq(800,600)[T\times_X Q]
\node p(1100,300)[P]

\arrow|b|[a`y;f]
\arrow|b|[y`x;g]
\arrow||/@{->}^<(0.3){\epsilon_p}/[gp`p;]
\arrow[geq`eq;\sigma']
\arrow[gfq`geq;\rho\times^Y g^\ast q]
\arrow||/@{->}^<(0.5){\ev_{p,q}}/[eq`p;]
\arrow||/@{-->}^<(0.5){g^\ast(\ev_{p,q})}/[geq`gp;]
\arrow[w`y;]
\arrow|m|[e`x;p^q]
\arrow|l|[gfq`a;(gf)^\ast q]
\arrow||/@{->}|!{(0,0);(800,0)}\hole|(0.6){g^\ast p}/[gp`y;]
\arrow|r|[p`x;p]
\arrow|a|[a`w;\rho]
\arrow||/@{->}|!{(300,400);(300,-300)}\hole|(0.6){\sigma}/[w`e;]
\arrow[geq`w;]
\arrow||/@{->}|!{(300,400);(1100,400)}\hole|(0.5){}/[eq`e;]
\efig
$$
\end{proof}
\begin{lemma}[Beck-Chevalley condition]
\label{lm:BCcond}
Let $\C$ be a locally cartesian closed $\infty$-category and let
$$
\bfig
\square<400,250>[D`C`A`B;h`k`f`g]
\place(70,180)[\angle]
\efig
$$ 
be a pullback square in $\C$. Then there are canonical natural equivalences
$$
\sum_k h^\ast\xrightarrow{\ \simeq \ } g^\ast\sum_f\quad\text{and}\quad f^\ast\prod_g \xrightarrow{\ \simeq\ } \prod_h k^\ast
$$
\end{lemma}
\begin{proof}
The first map being an equivalence at every $p\in\C_{/C}$ is a restatement of the pasting lemma for pullbacks. The result for dependent products follows from the one for dependent sums by taking right adjoints, since adjoints compose. 
\end{proof}

\subsection{Function extensionality}\label{sec:funext} In homotopy type theory, given types $X$ and $A$, and morphisms $f,g\colon X\rightarrow A$, there is a map
$$(f=_{A^X}g)\longrightarrow \prod_{x : X}(f(x)=_{A}g(x))$$ 
evaluating a path between $f$ and $g$ at each $x\of X$. The statement that this map is an equivalence (for all types $A,X$ and all $f,g\colon A\rightarrow X$) is known as \emph{function extensionality}. In our setting, function extensionality can be stated as follows.

\begin{proposition}[Function Extensionality]\label{prop:funext}
Let $\C$ be a locally cartesian closed $\infty$-category. Given $A,X\in\C$, let $\ev\colon A^X\times X\rightarrow A$ be the counit of the adjunction $(-)\times X\dashv (-)^X$ and form the pullback
$$
\bfig
\square<900,300>[Q`A`A^X\times A^X\times X`A\times A;`q`\Delta A`(\ev_1,\ev_2)]
\place(70,230)[\angle]
\efig
$$
Here $\ev_1$ (resp.~$\ev_2$) is the composite of the projection $A^X\times A^X\times X\rightarrow A^X\times X$ onto the first (resp.~second) and third components with the evaluation map. Let $\pr\colon A^X\times A^X\times X\rightarrow A^X\times A^X$ be the projection map. Then there is a canonical equivalence in $\C_{/A^X\times A^X}$
$$
\Delta (A^{X})\xrightarrow{\simeq} \prod_{\pr} q
$$
\end{proposition}
\begin{proof}
Let $k\colon E\rightarrow A^X\times A^X$ be an object in $\C_{/A^X\times A^X}$. By adjointness, there is a natural equivalence
$$\C_{/A^X\times A^X}\left(k,\prod_{\pr}q\right) \simeq\C_{/A^X\times A^X\times X}(k\times X,q)$$
By the description of hom-spaces in $\infty$-slice categories (see \cite[Lemma 5.5.5.12]{htt}) and since $Q$ is a pullback, we get a homotopy pullback square of $\infty$-groupoids
$$
\bfig
\square<1300,300>[\C_{/A^X\times A^X\times X}(k\times X,q)`\C(E\times X, A)`\ast`\C(E\times X, A\times A);``\C(E\times X,\Delta A)`(\ev_1,\ev_2)\circ (k\times X)]
\place(100,200)[\angle]
\efig
$$
But $\C(E,\Delta (A^X))\simeq\C(E\times X,\Delta A)\simeq \Delta_{\C(E\times X,A)}$,
which means that
$$
\C_{/A^X\times A^X\times X}(k\times X,q)\simeq\hofib_k(\C(E,\Delta(A^X)))\simeq\C_{/A^X\times A^X}(k,\Delta(A^X)),
$$
where the last equivalence is again \cite[Lemma 5.5.5.12]{htt}. We then get the needed composite natural equivalence
$$
\C_{/A^X\times A^X}\left(k,\prod_{\pr}q\right) \simeq\C_{/A^X\times A^X}(k,\Delta(A^X)).
$$
\end{proof}

\cref{prop:funext} can be promoted to a result about diagonals of dependent products. We now set up what we need to state this generalization of \cref{prop:funext}.

\smallskip
Let $p\colon E\rightarrow X$ be a map in $\C$ and let
$$
\bfig
\Vtriangle|alr|<400,250>[\left(\prod_X p\right)\times X`E`X;\epsilon`\pi`p]
\efig
$$
be the component of the counit of the adjunction $(-)\times X\dashv \prod_X$ at $p\in\C_{/X}$. Here $\pi$ is the projection map onto $X$. The projection map
$$
\left(\prod_X p\right)\times \left(\prod_X p\right)\times X\rightarrow X
$$
is the product object $\pi\times^X\pi$ in $\C_{/X}$. We can thus describe the product map $\epsilon\times^X\epsilon\colon\pi\times^X\pi\rightarrow p\times^X p$ in $\C_{/X}$ as the map over $X$ given by
$$(\epsilon_1,\epsilon_2)\colon \left(\prod_X p\right)\times \left(\prod_X p\right)\times X\rightarrow E\times_X E,$$
where $\epsilon_1$ (resp.~$\epsilon_2$) is the composite of the projection $$\left(\prod_X p\right)\times \left(\prod_X p\right)\times X\rightarrow \left(\prod_X p\right)\times X$$ onto the first (resp.~the second) and third components with the counit map. The pullback of $\Delta p$ along $\epsilon\times^X\epsilon$ in $\C_{/X}$ can be described as the pullback square
\begin{equation}
\label{eq:depfunextsq}
\bfig
\square<1300,300>[Q'`E`\left(\prod_X p\right)\times \left(\prod_X p\right)\times X`E\times_X E;`q'`\Delta p`(\epsilon_1,\epsilon_2)]
\place(100,200)[\angle]
\efig
\end{equation}
in $\C$ and $Q'$ can be naturally regarded as an object over $X$. 
\begin{proposition}[Dependent Function Extensionality]\label{prop:depfunext}
Let $\C$ be a locally cartesian closed $\infty$-category and let $p\colon E\rightarrow X$ be a map in $\C$. Construct $q'$ as in \eqref{eq:depfunextsq} and let $$\pr\colon \left(\prod_X p\right)\times \left(\prod_X p\right)\times X\rightarrow \left(\prod_X p\right)\times \left(\prod_X p\right)$$
be the projection map. Then there is a canonical equivalence in $\C_{/\left(\prod_X p\right)\times \left(\prod_X p\right)}$
$$
\Delta\left(\prod_X p\right)\stackrel{\simeq}{\longrightarrow} \prod_{\pr} q'
$$
\end{proposition}
\emph{Mutatis mutandis}, the proof is the same as for \cref{prop:funext}, so we omit it.
\begin{remark}
\label{rmk:localnatureoffunext}
If $\C$ is a locally cartesian closed $\infty$-category, then so is $\C_{/X}$ for any $X\in\C$. Thus, \cref{prop:funext} and \cref{prop:depfunext} hold true also in $\C_{/X}$ and give, for maps $p\colon E\rightarrow X$, $f\colon Y\rightarrow X$ and $q\colon M\rightarrow Y$ in $\C$, an alternative description of the diagonal of $p^f\in\C_{/X}$ and of $\Delta \left(\prod_f q\right)$ as a map in $\C_{/Y}$.
\end{remark}

\subsection{Contractibility}
\label{subsec:contractibility} 
We provide here a criterion for the existence and the uniqueness of extensions of one map along another one with the same domain. This result is linked to the notion of contractibility in $\C$.\par
Recall that an object $A\in\C$ is \emph{contractible} if the map $A\rightarrow 1$ is an equivalence. When we apply this definition to an object $p\in\C_{/X}$, this means that $p$ is contractible in $\C_{/X}$ exactly when, seen as a map in $\C$, it is an equivalence. Since equivalences in an $\infty$-topos form a local class of maps, we immediately get the following result.

\begin{lemma}
\label{lm:iscontrandpllbck}
Let $\E$ be an $\infty$-topos and let $f\colon Y\rightarrow X$ be an effective epimorphism in $\E$. For any $p\in\E_{/X}$, $f^\ast(p)\in\E_{/Y}$ is contractible if and only if $p$ is.
\end{lemma} 

The following lemma is a standard exercise in $2$-category theory since the notions of slice $\infty$-categories and of adjunctions between $\infty$-categories can be completely characterized in the $2$-category of $\infty$-categories --- see \cite[\S 3 and 4]{eleofinftycats}.

\begin{lemma}
\label{lm:adjonslice}
Let $\adj[F][G]{\C}{\D}$ be an adjunction and let $D\in\D$. Then there is an induced adjunction on slice categories $$\adj[\li{F}][\li{G}]{\C_{/GD}}{\D_{/D}}$$
where, for $p\in\C_{/GD}$ and $q\in\D_{/D}$, $\bar{F}(p)=\epsilon_D Fp$ and $\bar{G}(q)=Gq$. \qed
\end{lemma}

\begin{lemma}
\label{lm:auxlemmaforuniqueextalongfibers}
Let $p\colon D\rightarrow B\times C$ be a map in a locally cartesian closed $\infty$-category $\C$. Consider the map $q\colon E\rightarrow B\times C^B$ given by the pullback square
$$
\bfig
\square<600,300>[E`D`B\times C^B`B\times C;`q`p`(\pr_1,\ev)]
\place(70,230)[\angle]
\efig 
$$
Then there is an equivalence
$$
\left(\prod_{B}\sum_{B\times C\rightarrow B} p\right) \simeq \left(\sum_{C^B}\prod_{B\times C^B\rightarrow C^B}q\right)
$$
\end{lemma}
\begin{proof}
Let $\pr_B\colon B\times C\rightarrow B$ and $\pr_{C^B}\colon B\times C^B\rightarrow C^B$ be the projection maps. Note that $\Pi_{\pr_{C^B}}q$ is, by definition, a map 
$
\Pi_{\pr_{C^B}}q\colon\Sigma_{C^B}\Pi_{\pr_{C^B}}q\longrightarrow C^B.
$
On the other hand, we can see $p$ as a map $p\colon \Sigma_{\pr_B}p\rightarrow \pr_B$ in $\C_{/B}$. Setting $\alpha :=\Pi_B p$, we then get a map
$$
\alpha\colon\prod_{B}\sum_{\pr_{B}} p \longrightarrow \prod_B \pr_B = C^B.
$$
It is therefore sufficient to show that $\alpha\simeq \Pi_{\pr_{C^B}}q$ in $\C_{/C^B}$. Let $k\colon Z\rightarrow C^B$ be an object in $\C_{/C^B}$. Using \cref{lm:adjonslice} applied to the adjunction $$\adj[B\times (-)][\prod_B][1.5cm]{\C}{\C_{/B}}$$
we get
$
\C_{/C^B}(k,\alpha)\simeq\left(\C_{/B}\right)_{/\pr_{B}}\left(\kappa^\sharp,p\right).
$
Here $\kappa^\sharp$ is the composite $(\pr_1,\ev)(B\times k)$, seen as a map from $(B\times Z\xrightarrow{\pr_1} B)$ to $\pr_B$, and thus as an object in $\left(\C_{/B}\right)_{/\pr_{B}}$. Since $\left(\C_{/B}\right)_{/\pr_{B}}\simeq\C_{/B\times C}$ and using the definition of $q=(\pr_1,\ev)^\ast p$, we obtain
$$
\C_{/C^B}(k,\alpha)\simeq\C_{/B\times C}\left(\kappa^\sharp,p\right)=\C_{/B\times C}\left((\pr_1,\ev)(B\times k),p\right)\simeq 
$$
$$
\simeq\C_{/B\times C^B}(B\times k,q)=\C_{/B\times C^B}\left((\pr_{C^B})^\ast k,q\right)\simeq\C_{/C^B}\left(k,\prod_{\pr_{C^B}}q\right),
$$
whence $\alpha\simeq \Pi_{\pr_{C^B}}q$, as needed.
\end{proof}

Intuitively, the following result is about the existence of a unique extension of a map $f$ along another map $g$ in terms of unique extensions along the fibers of $g$. Taking fibers out of the picture, we get the following odd-looking statement.

\begin{proposition}[cf.~{\cite[Lemma 2.23]{locinhott}}]
\label{prop:uniqueextalongfibers}
Let $f\colon A\rightarrow C$ and $g\colon A\rightarrow B$ be two maps in a locally cartesian closed $\infty$-category $\C$. Form the following pullback squares in $\C$
$$
\bfig
\square(0,0)<700,300>[A\times C`B`A\times B\times C`B\times B;`(\pr_A,g\times C)`\Delta B`g\times \pr_B]
\place(100,200)[\angle]
\square(1600,0)<700,300>[B\times A`C`A\times B\times C`C\times C;`(\pr_A,B\times f)`\Delta C`f\times \pr_C]
\place(1700,200)[\angle]
\efig
$$
Consider the following object in $\C_{/B}$
$$
E:=\sum_{B\times C\rightarrow B}\left(\prod_{A\times B\times C\rightarrow B\times C}(\pr_A,B\times f)^{(\pr_A,g\times C)}\right)
$$
where the displayed internal hom is taken in $\C_{/A\times B\times C}$. Then the following hold.
\begin{itemize}
\item[(i)] If we let $f\colon C^B\rightarrow C^A$ be the composite $C^B\rightarrow 1\xrightarrow{f}C^A$, there is an equivalence 
\begin{equation}
\label{eq:extfg}
\prod_B E\simeq\sum_{C^B}(f,C^g)^\ast\left(\Delta(C^A)\right)
\end{equation} 
\item[(ii)] The space of global elements of the right-hand side in \eqref{eq:extfg} is equivalent to the space $\Ext(f,g)$ of extensions of $f$ along $g$. In particular, if $\prod_B E$ is contractible in $\C_{/B}$, then there is a unique dotted extension in 
$$
\bfig
\node a(0,0)[A]
\node b(0,-300)[B]
\node c(300,0)[C]
\arrow|l|[a`b;g]
\arrow|a|[a`c;f]
\arrow/-->/[b`c;]
\efig
$$
\end{itemize}
\begin{proof}
We start by proving the first claim. We have
$$
(\pr_A,B\times f)^{(\pr_A,g\times C)}=\prod_{(\pr_A,g\times C)}(\pr_A,g\times C)^\ast (\pr_A, B\times f)
$$
Since $(\pr_A,B\times f)=(f\times \pr_C)^\ast (\Delta C)$ and $(f\times \pr_C)(\pr_A,g\times C)=f\times C$, we get that $(\pr_A,g\times C)^\ast (\pr_A, B\times f)=(\id_A,f)\colon A\rightarrow A\times C.$ Therefore, letting $\pr_{B\times C}\colon A\times B\times C\rightarrow B\times C$ be the projection map, we have
$$
\prod_{\pr_{B\times C}}(\pr_A,B\times f)^{(\pr_A,g\times C)}=\prod_{\pr_{B\times C}}\left(\prod_{(\pr_A,g\times C)}(\id_A,f)\right)=\prod_{g\times C}(\id_A,f)
$$
Using \cref{lm:auxlemmaforuniqueextalongfibers}, we then get
$$
\prod_B E=\prod_B\sum_{B\times C\rightarrow B}\prod_{g\times C}(\id_A,f)\simeq
\sum_{C^B}\prod_{\pr_{C^B}}(\pr_{1},\ev)^\ast\left(\prod_{g\times C}(\id_A,f)\right)=: E'
$$
where $\pr_{C^B}\colon B\times C^B\rightarrow C^B$ is the projection map. There are pullback squares
$$
\bfig
\node acb(0,0)[A\times C^B]
\node ac(950,0)[A\times C]
\node bcb(0,-250)[B\times C^B]
\node bc(950,-250)[B\times C]

\node a(1600,0)[A]
\node c(2200,0)[C]
\node ac1(1600,-250)[A\times C]
\node cc(2200,-250)[C\times C]

\arrow|a|[acb`ac;(\id_A,\ \ev(g\times C^B))]
\arrow|l|[acb`bcb;g\times C^B]
\arrow|r|[ac`bc;g\times C]
\arrow|b|[bcb`bc;(\pr_1,\ev)]

\arrow|a|[a`c;f]
\arrow|l|[a`ac1;(\id_A,f)]
\arrow|r|[c`cc;\Delta C]
\arrow|b|[ac1`cc;f\times C]

\place(70,-70)[\angle]
\place(1670,-70)[\angle]
\efig
$$
Thus, using the Beck-Chevalley condition, we get
$$
E'\simeq \sum_{C^B}\prod_{\pr_{C^B}}\prod_{g\times C^B}\left((f\times C)(\id_A,\ \ev(g\times C^B))\right)^\ast(\Delta C)\simeq
$$
$$
\simeq \sum_{C^B}\prod_{A\times C^B\rightarrow C^B}\left((f\times C)(\id_A,\ \ev(g\times C^B))\right)^\ast(\Delta C)\simeq
$$
$$
\simeq \sum_{C^B}\prod_{A\times C^B\rightarrow C^B} \left(\ev(A\times (f,C^g))\right)^\ast(\Delta C)=:E''
$$
where the last equivalence is due to the fact that $(f\times C)(\id_A,\ \ev(g\times C^B))$ is equal to the composite map $\ev\circ \left(A\times (f,C^g)\right)$. Using the Beck-Chevalley condition applied to the pullback square
$$
\bfig
\node acb(0,0)[A\times C^B]
\node acaca(1000,0)[A\times C^A\times C^A]
\node cb(0,-250)[C^B]
\node caca(1000,-250)[C^A\times C^A]

\arrow|a|[acb`acaca;A\times (f,C^g)]
\arrow|l|[acb`cb;\pr_2]
\arrow|r|[acaca`caca;\pr_2]
\arrow|b|[cb`caca;(f,C^g)]

\place(70,-70)[\angle]
\efig
$$
we further deduce that
$$
E''=\sum_{C^B}\prod_{A\times C^B\rightarrow C^B}(A\times (f,C^g))^\ast (\ev^\ast(\Delta C))\simeq
$$
$$
\simeq \sum_{C^B}(f,C^g)^\ast\left(\prod_{\pr_2}\ev^\ast(\Delta C)\right)\simeq\sum_{C^B}(f,C^g)^\ast (\Delta(C^A))
$$
where the last equivalence is given by Function Extensionality.\par
For the second part, $P:=\sum_{C^B}(f,C^g)^\ast (\Delta(C^A))$ is the pullback object of $C^g$ along $f\colon 1\rightarrow C^A$ and thus $\C(1,P)$ is the homotopy fiber of $\C(1,C^g)$ at $f\in\C(1,C^A)$. The latter homotopy fiber gives the needed space of extensions.
\end{proof}
\end{proposition} 
\bibliographystyle{amsalpha2} % (change according to your preference)
%%% ***   Set the bibliography file.   ***
\bibliography{Llocalization.bib}

\end{document}

%% file: commands1.tex
\newcommand{\adj}[1][]{\def\ArgI{#1}\adjRelayI}
\newcommand{\adjRelayI}[1][]{\def\ArgII{#1}\adjRelayII}
\newcommand{\adjRelayII}[3][2.2em]{\ensuremath{\SelectTips{cm}{10}\xymatrix@C=#1@1{{#2} \ar@<1ex>[r]^-{\ArgI}^-{}="1" & {#3} \ar@<1ex>[l]^-{\ArgII}^-{}="2" \ar@{}"1";"2"|(.3){\hbox{}}="7" \ar@{}"1";"2"|(.7){\hbox{}}="8" \ar@{|-} "8" ;"7"}}}

\theoremstyle{plain}
\newtheorem{theorem}{Theorem}[section]
\newtheorem*{theoremstar}{Theorem}
\newtheorem{proposition}[theorem]{Proposition}
\newtheorem{lemma}[theorem]{Lemma}
\newtheorem{corollary}[theorem]{Corollary}

\theoremstyle{definition}
\newtheorem{definition}[theorem]{Definition}

\newtheorem{notation}[theorem]{Notation}
\newtheorem{remark}[theorem]{Remark}

\newcommand{\C}{\mathscr{C}} % a category
\newcommand{\D}{\mathscr{D}} % a category

\newcommand{\E}{\mathcal{E}}

\newcommand{\M}{\mathcal{M}}

\newcommand{\U}{\mathcal{U}} % a universe (in whichever sense)
\renewcommand{\phi}{\varphi}

\newcommand{\cod}{\mathrm{cod}} % codomain
\newcommand{\ev}{\mathrm{ev}}  % evaluation
\newcommand{\name}[1]{{\ulcorner {#1} \urcorner}}
\newcommand{\id}{\mathrm{id}}
\newcommand{\li}[1]{\overline{#1}} % overline the argument
\newcommand{\ti}[1]{\widetilde{#1}} % tilde over the argument
\newbox\anglebox
\setbox\anglebox=\hbox{\xy \POS(75,0)\ar@{-} (0,0) \ar@{-} (75,75)\endxy}
\def\angle{\copy\anglebox} %pullback symbol for squares

\newcommand{\of}{\! : \!} % colon for appearing in the context
\DeclareMathOperator{\hofib}{hofib}
\DeclareMathOperator{\Eq}{Eq}
\DeclareMathOperator{\Ext}{Ext}

\DeclareMathOperator{\pr}{pr}

\DeclarePairedDelimiterX{\norm}[1]{\lVert}{\rVert}{#1}
\DeclarePairedDelimiterX{\abs}[1]{\lvert}{\rvert}{#1}

\makeatletter
\newbox\xrat@below
\newbox\xrat@above
\newcommand{\xrightarrowtail}[2][]{%
  \setbox\xrat@below=\hbox{\ensuremath{\scriptstyle #1}}%
  \setbox\xrat@above=\hbox{\ensuremath{\scriptstyle #2}}%
  \pgfmathsetlengthmacro{\xrat@len}{max(\wd\xrat@below,\wd\xrat@above)+.6em}%
  \mathrel{\tikz [>->,baseline=-.75ex]
                 \draw (0,0) -- node[below=-2pt] {\box\xrat@below}
                                node[above=-2pt] {\box\xrat@above}
                       (\xrat@len,0) ;}}
\makeatother